\documentclass[11pt,a4paper]{article}
\usepackage{amsmath,amsfonts,amssymb,amsthm,amscd,color,pdfsync,stmaryrd,mathrsfs}
\usepackage[english]{babel}
\usepackage{graphicx}
\usepackage{hyperref,pdfsync}

\newcommand{\C}{{\mathbb C}}
\newcommand{\N}{{\mathbb N}}
\newcommand{\R}{{\mathbb R}}

\newcommand{\mc}[1]{\mathcal{#1}}

\newcommand{\Tc}{\mc{T}}

\newcommand{\curl}{{\rm curl}\,}

\newcommand{\loc}{{\rm loc}\,}
\renewcommand{\div}{{\rm div}\,}

\newcommand{\supp}{\operatorname{supp\,}}
\renewcommand{\phi}{{\varphi}}

\newcommand{\pd}{\partial} \newcommand{\pa}{\partial}
\renewcommand{\d}{\, {\rm d}}

\let\al=\alpha

\let\e=\varepsilon
\let\z=\zeta

\let\om=\omega

\let\Om=\Omega

\let\f=\frac

\newtheorem{theorem}{Theorem}[section]
\newtheorem{proposition}[theorem]{Proposition}
\newtheorem{lemma}[theorem]{Lemma}

\theoremstyle{remark}
\newtheorem{remark}[theorem]{Remark}

\numberwithin{equation}{section}

\title{Uniqueness for the 2-D Euler equations on domains with corners}
\author{Christophe Lacave, Evelyne Miot, Chao Wang}

\def\adrese{
\begin{description}
\item[C. Lacave:] Universit\'e Paris-Diderot (Paris 7), Institut de Math\'ematiques de Jussieu - Paris Rive Gauche, UMR 7586 - CNRS, B\^atiment Sophie Germain, Case 7012, 75205 PARIS Cedex 13, France.\\
Email: \texttt{lacave@math.jussieu.fr}\\
Web page: \texttt{http://www.math.jussieu.fr/$\sim$lacave/}
\item[E. Miot:] Universit\'e Paris-Sud, D\'epartement de math\'ematiques, B\^atiment 425, 91405 Orsay, France.\\
Email: \texttt{Evelyne.Miot@math.u-psud.fr}\\
Web page: \texttt{http://math.u-psud.fr/$\sim$miot/}
\item[C. Wang:] Universit\'e Paris-Diderot (Paris 7), Institut de Math\'ematiques de Jussieu - Paris Rive Gauche, UMR 7586 - CNRS, B\^atiment Sophie Germain, Case 7012, 75205 PARIS Cedex 13, France.\\
Beijing International Center For Mathematical Research, Peking University, 100871, P. R. China.\\
Email: \texttt{wangc@math.jussieu.fr}
\end{description}
}

\begin{document}

\date{\today}
\maketitle

\begin{abstract}
We prove uniqueness of the solution of the Euler equations with bounded vorticity for bounded simply connected planar domains with corners forming acute angles. Our strategy consists in mapping such domains on the unit disk via a biholomorphism. We then establish log-Lipschitz regularity for the resulting push-forward of the velocity field, which leads to uniqueness thanks to a Gronwall estimate involving the Lagrangian trajectories on the unit disk. 

\end{abstract}

\section{Introduction}

In this paper, we consider the motion of an ideal incompressible fluid in a bounded open set $\Omega\subset \R^2$. The velocity of the fluid $u=(u_1(t,x_1,x_2),u_2(t,x_1,x_2))$ satisfies the 2-D Euler equations:
\begin{equation}\label{E}
\begin{cases}
\pa_t u+u\cdot \nabla u +\nabla p=0,\\
\div u=0,
\end{cases}
\end{equation}
where $p:\Omega\to \R$ denotes the pressure. We supplement \eqref{E} with the initial data and an impermeability condition at the boundary $\pa \Om$:
\begin{equation}\label{E2}
u(0,\cdot)=u_0(\cdot), \quad u\cdot n|_{\pa \Om}=0.
\end{equation}

A natural quantity in the setting of fluids is the vorticity, defined by
\[
\omega=\curl u= \pa_1 u_2-\pa_2 u_1,
\]
which satisfies the transport equation
\begin{equation*}
\pa_t \omega+u\cdot \nabla \omega =0,\quad \omega(0,\cdot)=\omega_0(\cdot)=\curl u_0(\cdot).
\end{equation*}
There is a large variety of results for the 2-D Euler equations in the case of smooth domains or in the full plane. Global existence and uniqueness of smooth solutions was obtained in 1933 by Wolibner  \cite{wolibner} (see also Kato \cite{Kato} and Temam \cite{Temam}). In \cite{yudo}, Yudovich established global existence and uniqueness of the weak solution with uniformly bounded and integrable vorticity (see also Bardos \cite{Bardos} for a different approach concerning the existence part). Yudovich's arguments actually yield uniqueness for a larger class of initial vorticities (with a logarithmic-type blow up of the $L^p$ norms as $p\to\infty$), see \cite{yudo2} and more recently \cite{Bernicot1,Bernicot2,Vishik}.  Global weak solutions exist with very few regularity assumptions on the vorticity \cite{Delort,DiPernaMajda} or on the boundary \cite{BDT, GL,taylor}, while the above-mentioned uniqueness results hold for $C^{1,1}$ boundaries and Yudovich-type data.  
In domains with obtuse angles, uniqueness was obtained by Lacave  \cite{CL} for single-signed and bounded  vorticity. Recently, a uniqueness result without any sign condition has been established by Bardos, Di Plinio and Temam \cite{BDT} when the domain is a rectangle.

{\it The purpose of the present paper is to establish uniqueness for the 2-D Euler equations with bounded vorticity for a class of domains $\Omega\notin C^{1,1}$, more precisely:}
\begin{description}
\item[(H)] $\Omega$ is a bounded and simply connected open subset of $\R^2$, with $\partial \Omega$ belonging to $C^{2,\alpha}$ ($\alpha>0$) except in a finite number of points $\{ x_{k}\}_{k=1,\dots,N}$ where $\partial \Omega$ is a corner of angle $\theta_{k}\in (0,\pi/2]$.
\end{description}
This means that $\Om$ is locally parameterized closed to $x_k$ by $\{ z= x_k+r e^{i\theta}, \ 0<\theta<\theta_k,\:0< r< r_0\}$ for some $r_0>0$.


Existence of a global weak solution, with $u\in L^\infty (\R_+;L^2(\Omega))$ and $\omega  \in L^\infty(\R_+\times \Omega)$, of the 2-D Euler equations \eqref{E}-\eqref{E2} was proved by Taylor \cite{taylor} for convex domains (see also \cite{BDT} for further refined results) and by G\'erard-Varet and Lacave \cite{GL} for more general (possibly not convex) domains.

Here we will establish uniqueness of the weak solution:
\begin{theorem}\label{thm:main}
Assume that $\Omega$ satisfies (H). Let  $u_{0}$ be a vector field such that
\begin{equation*}
\curl u_{0} \in L^\infty(\Om), \quad \div u_0=0,\quad u_0\cdot n|_{\pa \Om}=0 .
\end{equation*}
Then the Euler equations \eqref{E}-\eqref{E2} have a unique global weak solution  such that
\begin{equation*}
u\in L^\infty (\R_+;L^2({\Om})),\quad \curl u \in L^\infty(\mathbb{R}_+\times \Om).
\end{equation*}
\end{theorem}

We now  give a few words on the main techniques involved in the previous proofs by Yudovich \cite{yudo} of uniqueness and the new ingredients used in the present paper.

The usual argument uses the regularity of the velocity $u=\nabla^\perp \Delta^{-1}\omega$ when $\omega$ is a bounded function. When  $\partial\Omega\in C^{1,1}$, the Calder\'on-Zygmund inequality (see e.g. \cite{Stein}) implies that the solution of
\begin{equation}\label{LP}
\Delta \psi = \omega \text{ on } \Omega, \quad \psi \vert_{\partial\Omega} = 0
\end{equation}
verifies for $p\in (1,\infty)$:
\begin{equation}\label{CZ}
\| D^2 \psi \|_{L^p(\Omega)} \leq C(p,\Omega) \| \omega \|_{L^p(\Omega)}.
\end{equation}
Moreover, there exists $C(\Omega)$ independent of $p$ such that for any $p\in [2,\infty)$, $C(p,\Omega)=C(\Omega)p$ in the above inequality. Hence, if $\omega \in L^\infty(\R_+\times \Omega)$, the Calder\'on-Zygmund inequality yields that the velocity $u=\nabla^\perp \psi$ belongs to $L^{\infty}(\R_+;W^{1,p}(\Omega))$ for any $p<\infty$. For a Gronwall type argument, we would need an estimate in $W^{1,\infty}$, which we almost have. In his uniqueness proof in \cite{yudo}, Yudovich used the fact that $C(p,\Omega)=C(\Omega)p$ in a crucial way to close the Gronwall estimate.

\medskip

Concerning non-smooth domains, the Calder\'on-Zygmund inequality \eqref{CZ} does not hold anymore. One striking counter-example was constructed by Jerison and Kenig in \cite{kenig}: the authors exhibit a smooth function $\omega$, a bounded domain such that $\partial\Omega \in C^1$, for which the second derivative of $\psi$ (the solution of the Laplace problem \eqref{LP}) is not integrable. This remark is an important obstruction to apply this argument for proving uniqueness. 

\medskip

Even for non $C^{1,1}$ domain, the assumption (H) allows to apply elliptic theory for the Laplace problem \eqref{LP} in domains with corners. This problem has been extensively studied (see e.g. \cite{Kondra, grisvard,Mazya}), and the behavior of $\psi$ close to the corners is known. In particular, the velocity $u=\nabla^\perp \psi$ is not bounded if $\theta_k>\pi$, and if $\theta_k> \pi/2$, the velocity does not  belong to $\bigcap_{p<\infty} W^{1,p}(\Omega)$  (which is the required regularity for the Yudovich's argument). For large corners,  the sign condition in \cite{CL} enables to prove that the support of the vorticity never intersects the boundary, which is the place where the velocity is not regular (see \cite{CL} for more details). On the contrary, when $\theta_k\leq\pi/2$, it is proved in 
\cite{grisvard} that \eqref{CZ} holds for all $2\leq p<+\infty$ with a constant $C(p,\Omega)$ depending only on $p$ and $\Omega$.  However, the estimate \eqref{CZ} with $C(p,\Omega)=C(\Omega)p$ does not seem to be available in the literature. Bardos, Di Plinio and Temam \cite{BDT} noticed that the proof of Grisvard \cite{grisvard} yields $C(p,\Omega)= C(\Omega) p^2$, which is not sufficient to prove the uniqueness as Yudovich did. Thanks to a symmetry and reflexion argument (only valid for convex domains with angles $\theta_{k}=\pi/m_k$, $m_k\in \N, m_k\geq 2$), the authors establish a new BMO estimate yielding \eqref{CZ} with $C(p,\Omega)=C(\Omega)p$.

\medskip

Our strategy is quite different here since it does not rely on the estimate \eqref{CZ}. Instead  we first map $\Omega$ to the unit disk $D$ by a biholomorphism $\mathcal{T}$. Then we establish a log-Lipschitz  (also called almost-Lipschitz) type estimate for the push forward of $u$ to the unit disk:
\[
U(y,t):=D\mathcal{T}(\mathcal{T}^{-1}(y))u(\mathcal{T}^{-1}(y)),
\]
which leads to an alternative proof of uniqueness using the trajectories of the Lagrangian flow of $U$ as is done in e.g. \cite[Theo. 3.1, Chapter 2]{marchioro-pulvirenti}. We emphasize that in our proof no estimate for $\|\nabla u\|_{L^p}$, for $p\to +\infty$, nor the log-Lipschitz regularity for $u$ are needed. Only the log-Lipschitz regularity for $U$ is used.

\medskip

Let us mention that the assumption $\partial \Omega\in C^{2,\alpha}$ away from the corners is a restriction which comes from a complex analysis result (namely the Kellogg-Warschawski theorem, see Proposition \ref{prop T}). It would be more natural to assume $\partial \Omega\in C^{1,1}$ away from the corners, which should be obtain in future work. Another interesting direction would be  to check if our techniques can be apply to more general Yudovich data, i.e. extending \cite{Bernicot1,Bernicot2, Vishik,yudo2} (for instance) where $C(p,\Omega)=C(\Omega)p$ in \eqref{CZ} is needed.

\medskip

The remainder of this work is organized in four sections. In the following section, we introduce the Riemann mapping, namely the biholomorphism mapping $\Omega$ to the unit disk $D$ and we establish some estimates holding in the neighborhoods of the angles of $\Omega$.  We then recall the explicit formula of the Biot-Savart law (giving the velocity in terms of the vorticity) in terms of the  Riemann mapping. Section \ref{sect 3} is the central part of this paper: thanks to the previous explicit formula and to the estimates on the Riemann mapping, we establish a new log-Lipschitz estimate for $U$. As the proof of the log-Lipschitz estimate is rather technical and follows from the same idea as an $L^\infty$ estimate for $u$, we detail our techniques in a simpler setting and  we postpone the log-Lipschitz proof to the last section. The proof of Theorem \ref{thm:main} is performed in Section \ref{sect 4}, by means of a Gronwall estimate involving the flow trajectories.

\medskip

\textbf{Notations.} 
In the sequel,  $C$ will denote a constant depending only on the domain $\Omega$, the value of which can possibly change from a line to another.

\section{Conformal Mapping and Biot-Savart Law}

\subsection{Conformal mapping}
Let $\Omega$ be a bounded simply-connected open subset of $\R^2$. Identifying $\R^2$ with $\C$, the Riemann mapping theorem states that there exists a biholomorphism $\mathcal{T}$ mapping $\Om$ to the unit disk $D=B(0,1)$ and $\partial \Om$ to $\partial D$. In the same spirit of \cite[Theo. 2.1]{CL}, we prove the following
\begin{proposition} \label{prop T}
Assume that there exists $\alpha>0$ such that $\partial \Omega$ is $C^{2,\alpha}$ except at a finite number of points $x_{k}, k=1,\ldots,N,$ at which $\partial \Omega$ is a corner of angle $\theta_{k}$.  There exists $0<\delta< \frac{1}{3}\min_{i\neq j}(|x_{i}-x_{j}|, |\Tc(x_{i})-\Tc(x_{j})|)$, and there exist $K>1$ and $M>1$, depending only on $\Omega$ and $\delta$, such that
 \begin{itemize}
 \item for all $x\in \Omega\setminus (\cup_{k=1}^N B(x_{k},\delta))$, all $y\in D\setminus (\cup_{k=1}^N B(\Tc(x_{k}),\delta))$ and any $k=1,\dots, N$ we have:
 \begin{equation*}\begin{split}
&|  D\mathcal{T}(x) | +  |  D^2 \mathcal{T}(x) | \leq K, \\
 &|  D \mathcal{T}^{-1}(y) | \leq K ;
\end{split}
\end{equation*}
  \item for any $k =1,\dots, N$, all $x \in \Omega\cap B(x_{k},\delta)$, all $y\in D\cap B(\Tc(x_{k}),\delta)$  we have:
\begin{equation*}\begin{split}
&M^{-1} |x-x_k|^{\pi/\theta_{k}} \leq | \mathcal{T}(x)-\Tc(x_{k}) | \leq M |x-x_k|^{\pi/\theta_{k}},\\
&|  D\mathcal{T}(x) | \leq M |x-x_k|^{\pi/\theta_{k}-1},\qquad |  D^2 \mathcal{T}(x) | \leq M |x-x_k|^{\pi/\theta_{k}-2},\\
&M^{-1}|y- \Tc(x_{k}) |^{\theta_{k}/\pi} \leq |   \mathcal{T}^{-1}(y) -x_k | \leq M|y-\Tc(x_{k}) |^{\theta_{k}/\pi},\\
 &|  D \mathcal{T}^{-1}(y) | \leq M |y-\Tc(x_{k}) |^{\theta_{k}/\pi-1}.
\end{split}
\end{equation*}
\end{itemize}
\end{proposition}

\begin{proof}
Even if the estimates concerning $\Tc^{-1}$ and $D\Tc^{-1}$ can be found in \cite[Theo. 3.9]{Pomm}, we provide here a self-contained proof.

\medskip

For simplicity,  we first assume  that there is only one corner with vertex $x_1$ and angle $\theta_1$.
 Thus $\partial\Omega$ is smooth except at $x_1$. 
In this proof we identify $\R^2$ and $\C$ to write $\Tc'(z), \Tc''(z),z\in \C$ instead of $D\Tc(x), D^2\Tc(x)$.

We introduce $\varphi_1(z) = (z-x_{1})^{\pi/\theta_1}$ and $\widetilde \Omega =\varphi_1(\Omega)$ a $C^{2,\alpha}$ bounded open set. Indeed, $\varphi_1$ locally sends the corner to the half plane. Then,  $f:=  \varphi_1\circ \Tc^{-1}$ is  a Riemann mapping from $D$ to $\widetilde \Omega$. By the Kellogg-Warschawski Theorem (see \cite[Theo. 3.6]{Pomm}), we infer that $f$ is $C^2$ on $\overline{D}$. Moreover, there exists a positive $C_1$ such that
$$C_1^{-1}\leq |f'(\zeta)|\leq C_1,\quad \forall \zeta\in \overline{D}$$
(see e.g. \cite[Theo. 3.5]{Pomm}).
Therefore, differentiating the relation ${\rm Id} = \Tc \circ \varphi^{-1}_1\circ f$, we compute
\begin{equation}\label{derive}
1= f'(\zeta) (\varphi^{-1}_1)'(f(\zeta))\Tc'(\varphi_1^{-1}\circ f(\zeta)).
\end{equation}
Since $(\varphi^{-1}_1)'(f(\zeta))=\frac{\theta_1}{\pi} f(\zeta)^{\theta_1/\pi-1}=\frac{\theta_1}{\pi} (\varphi_1^{-1}\circ f(\zeta)-x_1)^{1-\pi/\theta_1}$ we finally obtain (with $x=\varphi_1^{-1}\circ f(\zeta)$)
\begin{equation}\label{holo}
C_2^{-1} |x-x_1|^{\pi/\theta_1-1}\leq |\Tc'(x)|\leq C_2|x-x_1|^{\pi/\theta_1-1},\quad \forall x\in \overline{\Omega}.
\end{equation}
In particular, it follows from the mean-value theorem that
$$|\Tc(x)-\Tc(x_1)|\leq C_2\sup_{u\in [x_1,x]}|u -x_{1}|^{\pi/\theta_1-1} |x -x_{1}|  \leq C_2|x-x_1|^{\pi/\theta_1},$$
which also implies that 
$$C_2^{\theta_1/\pi} |y-\Tc(x_1)|^{\theta_1/\pi} \leq | \Tc^{-1} (y) -x_1|.$$
Next, by \eqref{holo} and the above estimate we have (since $1-\pi/\theta_1<0$)
$$|(\Tc^{-1})'(y)|\leq C_2|\Tc^{-1}(y)-x_1|^{1-\pi/\theta_1}\leq C_3|y-\Tc(x_1)|^{\theta_1/\pi-1}.$$
The mean-value theorem yields
$$
|\Tc^{-1}(y)-x_1| =  |f(y) |^{\theta_1/\pi} =  |f(y) -f(\Tc(x_1)) |^{\theta_1/\pi}       \leq C_4|y-\Tc(x_1)|^{\theta_1/\pi},\quad \forall y\in \overline{D},
$$
which also implies
$$C_4^{\pi/\theta_1} |x-x_1|^{\pi/\theta_1}       \leq |\Tc(x) -\Tc(x_1)|,\quad \forall x\in \overline{\Omega}. $$
Finally, differentiating \eqref{derive} we obtain
$$
0= f'' (\varphi_1^{-1})'(f)\Tc'(\varphi_1^{-1}\circ f)+(f')^2(\varphi_1^{-1})''(f)\Tc'(\varphi_1^{-1}\circ f)+
[f'(\varphi_1^{-1})'(f)]^2\Tc''(\varphi_1^{-1}\circ f)
$$
and the various previous bounds imply (with $x=\varphi_1^{-1}\circ f(\zeta)$)
$$|\Tc''(x)|\leq C_5|x-x_1|^{\pi/\theta_1-2},\quad \forall x\in \overline{\Omega}.$$
The conclusion follows for the estimates in the neighborhood of $x_1$. The first part of the estimates (away from $B(x_1,\delta)$ for a fixed $\delta>0$) is easily deduced from the previous ones.

\medskip

When $\Omega$ has exactly two corners, we set $\varphi_1(z) = (z-x_{1})^{\pi/\theta_1}$, $\varphi_2(z) = (z-\varphi_1(x_{2}))^{\pi/\theta_2}$, $\widetilde \Omega = \varphi_2\circ \varphi_1(\Omega)$ and  $f=  \varphi_2\circ \varphi_1\circ \Tc^{-1}$. We note that $ \varphi_1(\Omega)$ has only one corner at the point $\varphi_1(x_{1})$ with angle $\theta_2$ (because conformal mappings preserve angles). Therefore, $\widetilde \Omega$ is a $C^{2,\alpha}$ bounded open set and the Kellogg-Warschawski Theorem can be applied on $f$. Hence we can follow exactly the same proof as in the case of one corner. Since $\varphi_1$  (resp. $\varphi_1^{-1}$) is smooth away from $x_1$ (resp. $\varphi_1(x_1)$) we obtain the  estimates of  Proposition \ref{prop T}. For $N$ corners, we proceed in the same way: we apply $\varphi_k$ which sends the angle $\theta_k$ to the half plane, passing to $N+1-k$ corners to $N-k$ corners. 
\end{proof}

\begin{remark}
This proposition is one key point of our analysis and will be used in several estimates. The divergence of $D\Tc^{-1}$ close to the corners is linked with the behavior of $\Delta^{-1}$, and the form of the divergence will be crucial. Alternatively, we claim that Theorem \ref{thm:main} holds for more general corners: let $\gamma$ a parametrization of $\partial \Omega$ (couterclockwise direction, i.e. ${\rm Ind}_{\gamma}(z)\in \{0,1\}$), then we can consider:
\begin{description}
\item[(H')] $\gamma$ is a $C^{2,\alpha}$ ($\alpha>0$) Jordan curve, except in a finite number of points $\{ x_{k}\}_{k=1,\dots,N}$ of parameter $\{ s_{k}\}_{k=1,\dots,N}$  where $\displaystyle \lim_{s\to 0,s>0} {\rm arg}(-\gamma'(s_k-s),\gamma'(s_k+s))=\theta_{k}\in (0,\pi/2]$.
\end{description}
Actually, in the previous proof of Proposition \ref{prop T},  we would also need  a compatibility condition on $\gamma''$ near $s_k$ so that $(\gamma(s)-x_k)^{\pi/\theta_k}$ is $C^{2,\alpha}$ close to $s_k$, in order to use the Kellogg-Warschawski Theorem. For simplicity, we will further assume that $\partial \Omega$ is locally a corner near $x_k$, namely that (H) is satisfied. 
\end{remark}

\begin{remark}
The assumption $\partial \Omega \in C^{2,\alpha}$ away the corners comes from the use of Kellogg-Warschawski Theorem to get estimates on the second-order derivatives, which will be useful in our analysis (in particular for the log-Lipschitz regularity, see Section \ref{sect 5}).
\end{remark}

\medskip

In this article, we only consider acute angles $0<\theta_{k}\leq \frac\pi2$. We set
$$\al_{k} := 1-\frac{\theta_{k}}{\pi},\quad \text{so that}\quad \frac{1}{2}\leq \alpha_{k}<1.$$
In particular, we infer from Proposition \ref{prop T} that
\begin{eqnarray}
&&|D\mathcal{T}(\mathcal{T}^{-1}(y))|\leq C |y-\Tc(x_{k})|^{\al_{k}},\quad \forall y\in B(\Tc(x_{k}),\delta)\cap D;\label{est:DT1}\\
&&|D^2\mathcal{T}(\Tc^{-1}(y))|\leq C|y-\Tc(x_{k})|^{2\alpha_k-1},\quad \forall y\in B(\Tc(x_{k}),\delta)\cap D;\label{est:DT1-bis}\\
&&|D\mathcal{T}(\mathcal{T}^{-1}(y))|\leq C,\quad \forall y\in  D;\label{est:DT2}
\end{eqnarray}
and
\begin{eqnarray}\label{est:orsay}
&&| D\Tc^{-1}(\Tc(x)) | \leq C |x-x_k|^{1-1/(1-\alpha_{k})}.
\end{eqnarray}
Indeed, in the neighborhood  of $\Tc(x_{k})$ we have:
\[
|D\mathcal{T}(\mathcal{T}^{-1}(y))|\leq M |\mathcal{T}^{-1}(y)-x_{k}|^{\pi/\theta_{k}-1}\leq M M^{\pi/\theta_{k}-1} |y-\Tc(x_{k})|^{1- \theta_{k}/\pi}
\]
because $\pi/\theta_{k}-1\geq 0$. Moreover
\[
|D^2\mathcal{T}(\mathcal{T}^{-1}(y))|\leq M |\mathcal{T}^{-1}(y)-x_{k}|^{\pi/\theta_{k}-2}\leq M M^{\pi/\theta_{k}-2} |y-\Tc(x_{k})|^{1- 2\theta_{k}/\pi}
\]
because $\pi/\theta_{k}-2\geq 0.$
Finally,
\begin{eqnarray*}
&&| D\Tc^{-1}(\Tc(x)) | \leq M |\Tc(x) -\Tc(x_{k}) |^{\theta_{k}/\pi-1} \leq M M^{-\theta_{k}/\pi+1} |x-x_k|^{1-\pi/\theta_{k}}.
\end{eqnarray*}

\begin{remark}\label{DTDTt}
Since $\Tc$ is holomorphic we have
\begin{equation*}
D\mathcal{T}(x)D\mathcal{T}^T(x)=\det (D\mathcal{T}(x)) \mathrm{Id}=|\det (D\mathcal{T}(x))|  \mathrm{Id}.
\end{equation*}
\end{remark}
\subsection{Biot-Savart Law}

In this subsection, we recall briefly the Biot-Savart law on bounded and simply connected open domains (for more details, we refer to \cite{CL} and references therein). Let $\Omega$ be such a domain. For any $\omega\in L^\infty(\Omega)$, there exists a unique vector field $u$ such that 
\[
u\in L^2(\Omega), \quad \div u =0 \text{ on }\Omega,\quad \curl u =\omega \text{ on }\Omega,\quad u\cdot  n =0 \text{ on }\partial \Omega.
\]

Moreover, this vector field is given in terms of $\omega$ by the explicit formula of the Biot-Savart law:
\begin{equation}\label{eq:BS-1}
u(x)=K_\Omega[\omega](x)= \int_{\Om}K_{\Om}(x,\tilde{x})\omega(\tilde{x})\d\tilde{x},
\end{equation}
with
\[
K_{\Om}(x,\tilde{x})=\f{1}{2\pi}D\mathcal{T}^T(x)\left(\f{(\mathcal{T}(x)-\mathcal{T}(\tilde{x}))^\bot}{|\mathcal{T}(x)-\mathcal{T}(\tilde{x})|^2}-\f{(\mathcal{T}(x)
-\mathcal{T}(\tilde{x})^*)^\bot}{|\mathcal{T}(x)-\mathcal{T}(\tilde{x})^*|^2}\right)= D\mathcal{T}^T(x) K_{D}(\mathcal{T}(x),\mathcal{T}(\tilde{x})) ,
\]
where $\Tc$ is any biholomorphism mapping $\Omega$ to $D=B(0,1)$ (see the previous subsection) and where $K_{D}$ is the Biot-Savart kernel for the unit disk:
\[
K_{D}(x,\tilde{x})=\f{1}{2\pi}\left(\f{x-\tilde{x}}{|x-\tilde{x}|^2}-\f{x-\tilde{x}^*}{|x-\tilde{x}^*|^2}\right)^{\perp}, \ x\neq \tilde x\in D.
\]
Here the star comes from the image method, i.e.
\[
z^* := \frac{z}{|z|^2},
\]
and $\begin{pmatrix} z_{1}\\z_{2}\end{pmatrix}^\perp=\begin{pmatrix} -z_{2}\\z_{1}\end{pmatrix}$ (we set $K_D(x,0)=\f{1}{2\pi}\f{x^\perp}{|x|^2}$).
For all $y\in D$, changing variable $z=\Tc(\tilde x)$ in the integral, one can rewrite the Biot-Savart law as follows:
\begin{equation}\label{eq:BS-2}
u(\Tc^{-1}(y))=\f{1}{2\pi}D\mathcal{T}^T(\Tc^{-1}(y))
\int_{D}\left(\f{(y-z)^\bot}{|y-z|^2}-\f{(y-z^*)^\bot}{|y-z^*|^2}\right)\omega(\mathcal{T}^{-1}(z))\det (D\mathcal{T}^{-1}(z)) \d z.
\end{equation}

\medskip

We next recall the following useful properties for the Biot-Savart Kernel on $D$:
\begin{lemma} There exists $C>0$ such that:
\begin{equation}\label{estimate:K1}
|K_{D}(y,z)|\leq \frac{C}{|y-z|}, \quad \forall (y,z)\in D^2,\  y\neq z,
\end{equation}
and
\begin{equation}\label{estimate:K2}
|K_{D}(y_1,z)-K_{D}(y_2,z)|\leq C \frac{|y_1-y_2|}{|y_1-z||y_2-z|}, \quad \forall (y_1,y_2)\in D^2,\  z\in D\setminus \{y_1,y_2\}.
\end{equation}
\end{lemma}
\begin{proof} 
For the first estimate, the case $z=0$ is obvious, and we compute for $z\neq0$:
\[
|y-z|\leq |y-z^* | + |z^*-z|.
\]
We denote by $r=|z|$ and $p_{D}(z^*)$ the orthogonal projection of $z^*$ on $\overline{D}$. Then, for any $z\in D\setminus \{0\}$, we have $|z^*|>1$ and $p_{D}(z^*)\in[z,z^*]$. We compute easily that
\[
|p_{D}(z^*)-z| = 1-r,\quad |z^*-p_{D}(z^*)| = \frac1r-1,
\]
so
\[
|z^*-z| = (1-r)+(\frac1r-1).
\]
We observe that, since $r>0$, we have that $1-r\leq  \frac1r-1$, hence
\[
|z^*-z|  \leq 2 (\frac1r-1)=2   d(z^*,D)\leq 2 |z^*-y|, \quad \forall y\in D.
\]
Therefore
\begin{equation}\label{z-z*}
|y-z|\leq 3|y-z^* | , \quad \forall (y,z)\in D\times (D\setminus \{0\})
\end{equation}
which implies that
\[
2\pi|K_{D}(y,z)|\leq \frac{1}{|y-z|}+\frac{1}{|y-z^*|}\leq \frac{4}{|y-z|},
\]
which ends the proof of \eqref{estimate:K1}.

The second estimate follows directly from \eqref{z-z*} and the following equality
\begin{equation}\label{ab}
\Bigl| \frac{a}{|a|^2}- \frac{b}{|b|^2} \Bigl|=  \frac{|a-b|}{|a||b|}
\end{equation}
which can be proved by squaring both sides.
\end{proof}

Moreover, we have the classical log-Lipschitz regularity:
\begin{lemma} \label{est:K3}
There exists $C>0$ such that for all  $(y_1,y_2)\in D^2$:
\begin{equation*}
\begin{split}
\int_{D} |K_D(y_1,z)-K_D(y_2,z)|\d z\leq Ch(|y_1-y_2|) ,\\
\int_{D} |K_D(z,y_{1})-K_D(z,y_2)|\d z\leq Ch(|y_1-y_2|),
\end{split}
\end{equation*}
where $h:\R_+\to \R_+$ is defined by
\begin{equation}\label{defi h}
h(r):=r(1+|\ln r|),\quad \forall r\geq 0.
\end{equation}
\end{lemma}

\begin{proof}
The first estimate is standard (see e.g. \cite[Lem. 3.1, Chapter 2]{marchioro-pulvirenti}), but we write this proof because it will be extended later for $\Omega$.

Let $y_{1}, y_{2} \in D$, let $d=|y_{1}-y_{2}|$ and $\tilde y=\frac{y_{1}+y_{2}}2$.

Then we separate the integral in two parts:
\[
 \int_{B(\tilde y,d)\cap D} |K_D(y_1,z)-K_D(y_2,z)|\d z + \int_{B(\tilde y,d)^c\cap D} |K_D(y_1,z)-K_D(y_2,z)|\d z.
\]
When $|\tilde y -z |\geq d$ then $|\tilde y -z | \leq \frac{d}2+ |y_{1}-z|\leq \frac{|\tilde y -z |}2+ |y_{1}-z|$ which implies that $|y_{1}-z|\geq \frac{|\tilde y -z |}2$.
In the first integral above, we use \eqref{estimate:K1} whereas in the second one we use \eqref{estimate:K2}:
\begin{eqnarray*}
\int_{D} |K_D(y_1,z)-K_D(y_2,z)|\d z &\leq&  \int_{B(\tilde y,d)} \frac{C\d z}{|y_{1}-z|}+ \int_{B(\tilde y,d)} \frac{C\d z}{|y_{2}-z|} +  \int_{B(\tilde y,d)^c\cap D} \frac{Cd \d z}{|y_{1}-z| |y_{2}-z|}\\
&\leq& 2C \int_{B(0,3d/2)} \frac{\d y}{|y|}+ Cd \int_{B(0,2)\setminus B(0,d)} \frac{4 \d y}{|y|^2}\\
&\leq & 6\pi  C d +8\pi C d(\ln 2- \ln d) \leq  C h(d),
\end{eqnarray*}
 which ends the proof for the first inequality.
 
 The second statement is original and will be used only at the end of the article (namely page \pageref{K3bis}). As before, we separate the integral in two parts, the first part of which can be treated exactly in the same way with \eqref{estimate:K1}:
 \begin{eqnarray*}
 \int_{B(\tilde y,d)\cap D} |K_D(z,y_1)-K_D(z,y_2)|\d z \leq  \int_{B(\tilde y,d)} \frac{C\d z}{|y_{1}-z|}+ \int_{B(\tilde y,d)} \frac{C\d z}{|y_{2}-z|} \leq 6\pi  C d.
 \end{eqnarray*}
Concerning the second part, we write:
\[
|K_D(z,y_1)-K_D(z,y_2)| \leq \frac{|y_1-y_2|}{|z-y_1||z-y_2|} +\frac{|y_1^*-y_2^*|}{|z-y_1^*||z-y_2^*|}=\frac{|y_1-y_2|}{|z-y_1||z-y_2|}+ \frac{|y_{1}-y_{2}|}{|y_{1}| |y_{2}||z-y_1^*||z-y_2^*|},
\]
where we have used \eqref{ab} three times. The first term in the right-hand side is exactly the one which was treated in the first statement. Without loss of generality, we can restrict to the case where $d<1/4$. So for $z\in B(\tilde y,d)^c\cap D$, we consider two cases:

\medskip

\textbf{Case 1: } $y_{1}$ or $y_{2}$ belongs to $B(0,1/4)$.

As $d<1/4$, it implies that both $y_{1}$ and $y_{2}$ belong to $B(0,1/2)$, hence we have for $i=1,2$:
\[
|y_{i}| |z-y_{i}^*|= \Big| |y_{i}| z - \frac{y_{i}}{|y_{i}|}\Big| \geq 1- |y_{i}| |z| \geq \frac12 
\]
because $z\in D$. Therefore, we conclude easily:
 \begin{eqnarray*}
 \int_{B(\tilde y,d)^c\cap D} |K_D(z,y_1)-K_D(z,y_2)|\d z \leq  d \int_{B(\tilde y,2)\setminus B(\tilde y,d)} \Bigl(\frac{4}{|z-\tilde y|^2}+4\Bigl) \d z \leq Cd(1+|\ln d|).
 \end{eqnarray*}
\medskip

\textbf{Case 2: } $y_{1}$ and $y_{2}$ belongs to $B(0,1/4)^c$.

By \eqref{z-z*}, we write that
\[
|y_{i}| |z-y_{i}^*| \geq \frac14 \frac13 |z-y_{i}| \geq  \frac14 \frac13 \frac12 |z-\tilde y|
\]
because $|\tilde y-z|\geq d$ (see the proof of the first statement). So, the conclusion  is as in the first statement:
 \begin{eqnarray*}
 \int_{B(\tilde y,d)^c\cap D} |K_D(z,y_1)-K_D(z,y_2)|\d z \leq  d \int_{B(\tilde y,2)\setminus B(\tilde y,d)} \frac{(4+24^2)\d z}{|z-\tilde y|^2} \leq Cd(1+|\ln d|).
 \end{eqnarray*}
The lemma is proved.
 \end{proof}

We conclude this section by recalling the following property,  which expresses the fact that the velocity field $(D\Tc u)\circ \Tc^{-1}$ is tangent to the boundary $\partial D$:
\begin{lemma}
\label{est:K4}
Let $y$ such that $|y|=1$. Then 
$$K_D(y,z)\cdot y=0,\quad \forall z\in D.$$
\end{lemma}
\begin{proof}
This can be proved from the following basic computation:
\begin{equation*}
\begin{split}
2\pi K_D(y,z)\cdot y =& \frac{-z^\perp \cdot y}{|y-z|^2}+\frac{z^{*\perp} \cdot y}{|y-z^*|^2}= z^\perp \cdot y \frac{(|y|^2-1)(\frac1{|z|^2}-1)}{|y-z|^2|y-z^*|^2},\quad \forall y, z \in \R^2, \ z\notin \{ y,  y^*\}.
\end{split}
\end{equation*}
\end{proof}

\section{$L^\infty$, $W^{1,p}$ and log-Lipschitz regularity}\label{sect 3}

As mentioned in the introduction,  it is classical in elliptic theory on domains with acute corners that if $\omega \in L^p(\Omega)$ then $u\in W^{1,p}(\Omega)$ for any $p\in [1,\infty)$, where $u$ is given by \eqref{eq:BS-1} (see e.g. \cite[Theo. 4.4.4.13]{grisvard}). For sake of 
self-containedness, we include in the proof of the proposition below the precise references stating that all the $W^{1,p}$ norms can be controlled by the $L^\infty$ norm of $\omega$. In particular, since a bounded vorticity solution $\omega(t,\cdot)$ of the Euler equations satisfies $\|\omega(t,\cdot)\|_{L^\infty(\Om)}\leq \|\omega(0,\cdot)\|_{L^\infty(\Om)} $ this will provide uniform in time estimates for the $W^{1,p}$ norms of the corresponding velocity $u(t,\cdot)$.

\begin{proposition} \label{prop:W1p-bound}
Let $\Om$ satisfy (H).  For any $p\in (1,\infty)$, there exists a constant $C(p,\Omega)$ depending on $\Om$ and $p$ such that for any  $u$ given by \eqref{eq:BS-1} with $\omega \in L^\infty(\Om)$, 
\[
\|u\|_{W^{1,p}(\Om)}\leq C(p,\Omega)\|\omega\|_{L^\infty(\Om)} .
\]
\end{proposition}
\begin{proof}In this proof the notation $K_j$, $j\in \N$ refers to constants depending only on $p$ and $\Omega$.

As $\omega \in L^\infty(\Omega)$ then Theorem 4.4.4.13 in \cite{grisvard} states that the solution of the elliptic problem \eqref{LP} verifies $\psi \in W^{2,p}(\Omega)$ for any $p\in [1,\infty)$. Hence, Theorem 4.3.2.4 for $p\neq 2$ and Theorem 4.3.1.4 for $p=2$ imply that
\begin{equation}\label{ineq:sam}
\|\psi \|_{W^{2,p}(\Om)}\leq K_1 (\|\omega\|_{L^p(\Om)}+\|\psi \|_{W^{1,p}(\Om)}) .
\end{equation} We assume first that $p=2$.
By performing an energy estimate on \eqref{LP} together with Cauchy-Schwarz, H\"older and Poincar\'e inequalities we have
\[
\| \nabla \psi \|_{L^2(\Om)}^2\leq \| \omega \|_{L^2(\Om)} \| \psi \|_{L^2(\Om)} \leq K_2\| \omega \|_{L^\infty(\Om)} \|\nabla \psi \|_{L^2(\Om)}
\]
so that $\|\psi \|_{H^{1}(\Om)}\leq K_3\| \nabla \psi \|_{L^2(\Om)}\leq K_3 K_2 \| \omega \|_{L^\infty(\Om)}$. Coming back to \eqref{ineq:sam} we obtain the estimate:
\[
\|\psi \|_{H^2(\Om)}\leq  K_4 \|\omega\|_{L^\infty(\Om)} .
\]

Then we assume that $p\neq 2$. Using first the embedding of $H^1(\Om)$ in $L^p(\Om)$, then the previous estimate for $p=2$, we infer that
\[
\|\psi \|_{W^{1,p}(\Om)}\leq K_5\|\psi \|_{H^{2}(\Om)}\leq K_6\|\omega\|_{L^\infty(\Om)},
\]
so that, using again \eqref{ineq:sam} we obtain
\[
\|\psi \|_{W^{2,p}(\Om)}\leq  K_6 \|\omega\|_{L^\infty(\Om)} .
\]
 Since $u=\nabla^\perp \psi$, the proof is complete.
\end{proof}
In particular,  the previous estimates also yield that $u$ is bounded, with $\|u\|_{L^\infty(\Om)}\leq  C(\Om)\|\omega\|_{L^\infty(\Om)} $. However, in order to exemplify in a simpler setting our techniques, which will be used for the log-Lipschitz regularity, we present below an alternative proof of this fact. It is based only on the Biot-Savart law, without using theorems from the elliptic theory in domains with corners.  In the formula \eqref{eq:BS-2}, we recognize a Biot-Savart  type law in $D$ corresponding to a vorticity given by
\[ \tilde{\omega}(z)= \omega(\mathcal{T}^{-1}(z))\det (D\mathcal{T}^{-1}(z)) \leq C \|\omega\|_{L^\infty} |z-\Tc(x_{k})|^{-2\al_{k}}.\]
In view of \eqref{estimate:K1}  a natural assumption in order  to show that $K_D[\tilde{\omega}]$ is uniformly bounded is that  $\tilde{\omega}$ belongs to $L^q$ for some $q>2$. However,  $\tilde{\omega}$ can possibly not belong to $L^q$, for none of $q>2$, because $\alpha_{k}\geq 1/2$. The main idea of the following proposition is to take advantage of $D\Tc^T(\Tc^{-1}(y))  = \mathcal{O}(|y-\Tc(x_{k})|^{\al_{k}})$ in \eqref{eq:BS-2} to obtain the following uniform estimate on $u$.

\begin{proposition} \label{prop:uniform-bound}
Let $\Om$ satisfy  (H). Let $u$ be a field given by \eqref{eq:BS-1}, with $\omega \in L^\infty(\Om)$.  Then there exists a constant $C$ depending only on $\Om$ such that
\[
\|u\|_{L^\infty(\Om)}\leq C\|\omega\|_{L^\infty(\Om)} .
\]
\end{proposition}
\begin{proof}
After changing variable, we want to prove that for all $y\in D$
\begin{equation*}
|u(\mathcal{T}^{-1}(y))| \leq |D\mathcal{T}(\mathcal{T}^{-1}(y))| \int_{D} \left|K_D(y,z)\right||\omega(\mathcal{T}^{-1}(z))|\det(D\mathcal{T}^{-1}(z))\d z \leq C\| \omega\|_{L^\infty(\Omega)}.
\end{equation*}

Let $0<\delta<1$ be defined in Proposition \ref{prop T} and \eqref{est:DT1}.

\medskip

\textbf{Case 1: }$y\notin \cup_{k=1}^N B(\Tc(x_{k}),\delta/2)$.

By \eqref{est:DT2} we have $|D\mathcal{T}(\mathcal{T}^{-1}(y))|\leq C,$ so that
\begin{equation*}
\begin{split}
|u(\mathcal{T}^{-1}(y))|\leq & C \int_{D} \left|K_D(y,z)\right||\omega(\mathcal{T}^{-1}(z))|\det(D\mathcal{T}^{-1}(z))\d z\\
\leq & C \sum_{k=1}^N \int_{D\cap B(\Tc(x_{k}),\delta/4)} \left|K_D(y,z)\right||\omega(\mathcal{T}^{-1}(z))|\det(D\mathcal{T}^{-1}(z))\d z\\
&+C \int_{D\setminus \cup_{k=1}^N B(\Tc(x_{k}),\delta/4)} \left|K_D(y,z)\right||\omega(\mathcal{T}^{-1}(z))|\det(D\mathcal{T}^{-1}(z))\d z\\
\leq & I_1+I_2.
\end{split}
\end{equation*}
Observe that if $y \notin \cup_{k=1}^N B(\Tc(x_{k}),\delta/2)$ and $z\in B(\Tc(x_{k}),\delta/4)$ then $|y-z|\geq \delta/4$. Therefore by \eqref{estimate:K1} we obtain
$$I_1\leq \frac{4CN}{\delta} \int_{D}|\omega(\mathcal{T}^{-1}(z))| \det(D\mathcal{T}^{-1}(z))\d z= C\|\omega\|_{L^1(\Omega)},$$
where we have changed variable back.

Next, using that $\det(D\mathcal{T}^{-1}(z))$ is bounded for $z\in D\setminus \cup_{k=1}^N B(\Tc(x_{k}),\delta/4)$ we get
\begin{equation*}
\begin{split}
I_2&\leq C \|\omega\|_{L^\infty(\Omega)} \int_{D} \left|K_D(y,z)\right|\d z,
\end{split}
\end{equation*}
therefore by \eqref{estimate:K1} we get
\begin{equation*}
I_2\leq C \|\omega\|_{L^\infty(\Omega)}.
\end{equation*}

Finally, we have proved that
\begin{equation*}
|u(T^{-1}(y))|
\leq C \|\omega\|_{L^\infty(\Omega)},\quad \forall y\in D\setminus \cup_{k=1}^N B(\Tc(x_{k}),\delta/2).
\end{equation*}
\medskip

\textbf{Case 2: }$y\in B(\Tc(x_{k}),\delta/2)$ for some $k\in\{1,\ldots,N\}.$

Then we use \eqref{est:DT1} to write
\begin{equation*}
\begin{split}
|u(\mathcal{T}^{-1}(y))|&\leq C |\Tc(x_{k})-y|^{\al_{k}}
 \int_{D\cap B(\Tc(x_{k}),\delta)} \left|K_D(y,z)\right||\omega(\mathcal{T}^{-1}(z))|\det(D\mathcal{T}^{-1}(z))\d z\\
&+C  |\Tc(x_{k})-y|^{\al_{k}} \int_{D\cap B(\Tc(x_{k}),\delta)^c} \left|K_D(y,z)\right||\omega(\mathcal{T}^{-1}(z))|\det(D\mathcal{T}^{-1}(z))\d z.
\end{split}
\end{equation*}
Arguing as before, using that for $y\in B(\Tc(x_{k}),\delta/2)$ and $z\in B(\Tc(x_{k}),\delta)^c$ we have $|z-y|\geq \delta/2$, so we readily estimate the second integral of the right-hand side by
$$C \|\omega\|_{L^1(\Omega)}.$$

Therefore, using the bound for $\det(D\mathcal{T}^{-1})$ provided by Proposition \ref{prop T} when $z\in B(\Tc(x_{k}),\delta)$ we can actually reduce the problem to prove that
 \begin{equation}\label{ineq:u-2}
|y_{k}-y|^{\alpha_k} \int_{D} \f{1}{|y-z|}|z-y_{k}|^{-2{\alpha_k}}\d z \leq C, \quad \textrm{for} \quad y\in B(y_{k},\delta/2).
\end{equation}

Writing that
\begin{eqnarray*}
|y-y_{k}|^{{\alpha_k}}\leq (|y-z|+|z-y_{k}|)^{\alpha_k} \leq (2 \max (|y-z|,|z-y_{k}|))^{\alpha_k}  =  2^{\alpha_k} \max (|y-z|^{\alpha_k},|z-y_{k}|^{\alpha_k})
\end{eqnarray*}
we prove \eqref{ineq:u-2} by computing
 \begin{equation*}
\begin{split}
  &|y_{k}-y|^{\alpha_k} \int_{D} \f{1}{|y-z|}|z-y_{k}|^{-2{\alpha_k}}\d z\\
&\leq 2^{\alpha_k}\int_{D \cap \{|z-y_{k}|\leq |y-z|\}}|y-z|^{{\alpha_k}-1}|z-y_{k}|^{-2{\alpha_k}}\d z
+2^{\alpha_k}\int_{D\cap \{|z-y_{k}|\geq |y-z|\}}|y-z|^{-1}|z-y_{k}|^{-{\alpha_k}}\d z\\
&\leq 2^{\alpha_k} \int_{ D}|z-y_{k}|^{-1-{\alpha_k}}\d z+2^{\alpha_k}\int_{D}|y-z|^{-1-{\alpha_k}}\d z\\
&\leq C,
\end{split}
\end{equation*}
because $\alpha_{k}<1$.

This yields the conclusion of Proposition \ref{prop:uniform-bound}.
\end{proof}

\begin{remark}By similar arguments one can also establish the estimate of Proposition \ref{prop:W1p-bound} for a range of small values of  $p\in(1,p_c)$.
\end{remark}

The following proposition concerning the log-Lipschitz type regularity for the  field $(D\mathcal{T}\circ \mathcal{T}^{-1})u\circ \mathcal{T}^{-1}$ is original and  will be crucial for the uniqueness result.
\begin{proposition}\label{log-lip}
Let  $\Om$ satisfy (H) and $u$ given by \eqref{eq:BS-1}, with $\omega \in L^\infty(\Om)$.  Then the map
$$U: y\in D\mapsto D\mathcal{T}(\mathcal{T}^{-1}(y))u(\mathcal{T}^{-1}(y))$$ is log-Lipschitz on $D$. More precisely, there exists a constant depending only on $\Omega$ such that
\begin{equation}\label{ineq:loglip}
|U(y_1)-U(y_2)|\leq C \|\omega\|_{L^\infty(\Omega)} h (|y_1-y_2|),\quad \forall (y_1,y_2)\in D^2,
\end{equation}
where $h:\R_+\to \R_+$ is defined in \eqref{defi h}.
\end{proposition}

The proof of Proposition \ref{log-lip} follows the same idea as the one of Proposition \ref{prop:uniform-bound}: we adapt the classical proof in smooth domains (namely, the first estimate of Lemma \ref{est:K3}) to the case where we have a finite number of corners. Close to the corners, we use the precise behavior of $D\Tc$ (see Proposition \ref{prop T}) in order to counterbalance the divergence of $\det( D\Tc^{-1})$. We postpone the full proof in Section \ref{sect 5}. We stress that we are not able to establish such a log-Lipschitz regularity for $u$.

\section{Proof of Theorem \ref{thm:main}}\label{sect 4}
Let $\Om$ satisfy (H). Let $\omega_{0}\in L^\infty(\Om)$ and $u_0$ satisfying
\begin{equation*}
u_0 \in L^2(\Omega),  \quad  \curl u_0 =\omega_{0}, \quad \div u_0 = 0 , \quad u_0 \cdot n \vert_{\pa \Omega} = 0.
\end{equation*}
By the main result in \cite{taylor} for convex domains and \cite{GL} for general domains, we know that there exists a global weak solution on $\Om$ such that
\begin{equation*}
u\in L^\infty (\R_+;L^2({\Om})),\quad \omega \in L^\infty(\mathbb{R}_+\times \Om).
\end{equation*}
Let us make a short comment about the tangency condition in \cite{GL}. Without any assumption about the regularity of the boundary, the tangency condition in \cite{GL} reads as
\begin{equation}\label{imperm}
 \int_\Omega u \cdot h = 0  \quad \text{for all } h \in G(\Omega):=\{w\in L^2(\Omega) \ : \ w=\nabla p, \ \text{ for some } p\in H^1_{\loc}(\Omega)\}.
\end{equation}
However, in our case, the assumption (H) and the $W^{1,p}$ estimate of Proposition \ref{prop:W1p-bound} imply that $u(t,\cdot)$ is continuous up to the boundary of $\Om$. With such a regularity, the condition \eqref{imperm} amounts to $u(t,\cdot)\cdot n=0$  on $\partial \Om$ except at the corners. Noticing that the normal vectors on each side at a given corner span $\R^2$, we actually have $u(t,\cdot)=0$ at each corner.   

The velocity field $u$ satisfies the regularity properties of Section \ref{sect 3} (Propositions \ref{prop:W1p-bound}, \ref{prop:uniform-bound} and \ref{log-lip}). In the next subsection we introduce the Lagrangian flow associated to $u$ before turning to the proof of uniqueness.

\subsection{The Lagrangian flow}

Let $(u,\omega)$ be as above. The purpose of this subsection is to establish the following
\begin{proposition}\label{prop:lagrangian-formulation}
There exists $X:\R_+\times \Omega\to \Omega$ such that for almost every $x\in \Omega$, $t\mapsto X(t,x)$ belongs to $W^{1,\infty}(\R_+,\Omega)$ and satisfies 
\begin{equation*}
 \displaystyle X(t,x)=x+\int_0^t u(s,X(s,x))\,ds,\quad \forall t\in \R_+.
\end{equation*}Moreover, for all $t\in \R_+$ the map $X(t,\cdot)$ preserves the Lebesgue's measure on $\Omega$.

Finally, we have $$\omega(t)=X(t,\cdot)_\#\omega_0, \quad \text{for a.e. }t\in\R_+, $$
in the sense that for a.e. $t\in \R_+$ we have $\int_\Omega \omega(t,x)\varphi(x)\,dx=\int_\Omega \omega_0(x)\varphi(X(t,x))\,dx$ for all $\varphi\in C_c(\Omega)$.
\end{proposition}

For smooth domains $\Omega$ and for velocity fields $u$ satisfying the regularity properties of Propositions  \ref{prop:W1p-bound} and  \ref{prop:uniform-bound} this was proved by DiPerna and Lions \cite{dip-lions} (p. 546). To the author's knowledge, such results are not available in the literature for general non-smooth domains. 

In order to prove  Proposition \ref{prop:lagrangian-formulation} we will proceed in three steps: first we will establish in Lemma \ref{lemma:existence-flow} the existence of the flow $X$ associated to $u$ such that the flow trajectories never reach the boundary. Then we will establish in Lemma \ref{lemma:transport-full}  that the extension of $\omega$ outside $\Omega$ satisfies a linear transport equation on $\R_+\times \R^2$ with a velocity field satisfying the usual assumptions ensuring uniqueness of the solution. We will finally conclude that $\omega$ is constant along the flow trajectories.

\medskip

\begin{lemma}\label{lemma:existence-flow}
There exists $X:\R_+\times \Omega\to \Omega$ such that for every $x\in \Omega$, $t\mapsto X(t,x)$ is the unique function belonging to $W^{1,\infty}(\R_+,\Omega)$ and satisfying
\begin{equation*}
 \displaystyle X(t,x)=x+\int_0^t u(s,X(s,x))\,ds,\quad \forall t\in \R_+.
\end{equation*}Moreover, for all $t\in \R_+$ the map $X(t,\cdot)$ preserves the Lebesgue's measure on $\Omega$.
\end{lemma}

\begin{proof}
In view of Proposition  \ref{prop:uniform-bound} (implying in particular that $U\in L^{\infty}(\R_+\times D)$) and Proposition \ref{log-lip}, a classical extension of Cauchy-Lipschitz theorem (see e.g. \cite[Chapter 2]{marchioro-pulvirenti}) implies that for any $y\in D$ there exists $t(y)>0$ and a unique curve  $t\mapsto {Y}(t,y)$ in $W^{1,\infty}([0,t(y))$ such that
\begin{equation*}
 \displaystyle Y(t,y)=y+\int_0^t U(s,Y(s,y))\,ds,\quad \forall t\in [0,t(y)).
\end{equation*}
Here $t(y)>0$ is the maximal time of existence, namely $Y(t,y)\in D$ on $[0,t(y))$ and $|Y(t(y),y)|=1$ if $t(y)<\infty$. Moreover since $U\in L^\infty$ we have 
\begin{equation*}
 \displaystyle \frac{d}{dt}{Y}(t,y)=U(t,{Y}(t,y)),\quad \text{for a.e. } t\in [0,t(y)).
\end{equation*}
We next show that $t(y)=+\infty$: let assume that $t(y)<+\infty$. Recalling that $U\big(t,\frac{Y(t,y)}{|Y(t,y)|}\big)\cdot Y(t,y)=0$ (see Lemma \ref{est:K4}), we get for a.e. $t\in [0,t(y))$ sufficiently close to $t(y)$, say on $[t_*,t(y))$ (so that $Y(t,y)\neq 0$)
\begin{equation*}
\frac{\d}{\d t}|Y(t,y)|=\Big[ U(t,Y(t,y))-U\big(t,\frac{Y(t,y)}{|Y(t,y)|}\big)\Big]\cdot \frac{Y(t,y)}{|Y(t,y)|}
\end{equation*}
so that, by Proposition \ref{log-lip},
\begin{equation*}
\left|\frac{d}{dt}|Y(t,y)|\right|\leq Ch(1-|Y(t,y)|).
\end{equation*}
Integrating the Gronwall-type inequality above yields
$$1-|Y(t,y)|\geq e^{1-\exp{(C[t-t_\ast])}}(1-|Y(t_*,y)|)^{\exp(C[t-t_\ast])},\quad \forall t\in[t_{*},t(y)).$$
 Since $|Y(t_*,y)|<1$ we are led to a contradiction, and the claim that $Y(t,y)\in D, \forall t\in \R_+$, follows.

Next, we introduce the map
$$X_t=\Tc^{-1}\circ Y_t\circ \Tc$$
($f_t$ denotes the map $f(t,\cdot)$). By definition of $\Tc$, for all $x\in \Om$, the previous result implies that $ X(t,x)\in\Omega$ for all $t\in \R_+$. Moreover we compute for a.e. $t\in \R_+$
\begin{equation*}
\begin{split}
\frac{\d}{\d t}\Big(\Tc^{-1}(Y(t,\Tc(x))) \Big) = &D\Tc^{-1}(Y(t,\Tc(x))) \frac{\d}{\d t}\Big(Y(t,\Tc(x)) \Big)=D\Tc^{-1}(Y(t,\Tc(x))) U(t,Y(t,\Tc(x)))\\
&= u(t,\Tc^{-1}(Y(t,\Tc(x)))) = u(t,X(t,x)).
\end{split}
\end{equation*}
Finally, the map ${X}(t,\cdot)$ preserves the Lebesgue's measure on $\Omega$ for all $t\in \R_+$ because  $u$ is divergence free.  
\end{proof}

\medskip

We now turn to the second step.
Since $\om \in L^\infty(\R_+\times \Omega)$, Proposition \ref{prop:uniform-bound} implies that
\begin{equation*}
u \in L^\infty(\R_+ \times \Om),
\end{equation*}
and from Proposition \ref{prop:W1p-bound} that for $p\geq 1$,
\begin{equation*}
u\in L^\infty(\R_+,W^{1,p}(\Om)).  
\end{equation*}

Let us consider the stream function $\psi$ of $u$, namely the function verifying:
\[
u=\nabla^\perp \psi \text{ in } \Om, \quad \psi =0 \text{ on } \partial \Omega.
\]
By the Poincar\'e inequality, we have 
\begin{equation}\label{psi W2p}
\psi \in L^\infty(\R_+,W^{2,p}(\Om)).  
\end{equation}

As $\Om$ verifies (H), we readily check that $\Om$ verifies the {Uniform Cone Condition} (see \cite[Par. 4.8]{Adams} for the precise definition). Therefore \cite[Theo. 5.28]{Adams} states that there exists a {simple} (2,p)-extension operator $E(p)$ from $W^{2,p}(\Om)$ to $W^{2,p}(\R^2)$, namely there exists $K(p)>0$ such that for any $v \in W^{2,p}(\Om)$
\[
E(p)v=v \text{ a.e. in } \Om, \quad \| E(p)v \|_{W^{2,p}(\R^2)} \leq K(p) \| v \|_{W^{2,p}(\Om)}.
\]

For the remainder of this subsection, we fix $p>1$   and $R$ large enough such that $\Om\subset B(0,R)$.
 Then we define $\chi$ a smooth cutoff function such that $\chi \equiv 1$ on $B(0,R)$ and $\chi \equiv 0$ on $B(0,R+1)$, and we set for a.e. $t\in \R_+$
\[
\bar \psi(t,\cdot) = \chi E(p)\psi(t,\cdot),\quad \bar u(t,\cdot) = \nabla^{\perp} \bar \psi(t,\cdot).
\]
Hence we have:
$$\bar u(t,\cdot) =u(t,\cdot) \text{ a.e. in }\Omega,$$
and
\begin{equation}\label{cond:u-div}
\div \bar u(t,\cdot)=0 \text{ a.e. in }\Omega.
\end{equation}
We also infer from \eqref{psi W2p} that
\begin{equation}\label{bar u}
\bar u \in L^\infty(\R_+;W^{1,p}(\R^2)).
\end{equation}

\medskip

We set next $\bar\omega$ to be the extension of $\omega$ by zero outside $\Om$. Then:
\begin{lemma}\label{lemma:transport-full}
The extension $\bar \omega$ is a solution of the transport equation:
\[ \partial_{t} \bar \omega + \bar u \cdot \nabla \bar \omega=0,\quad \bar \omega(0)=\overline{\omega_0}\]
in the sense of distributions in $\R_+\times \R^2$.
\end{lemma}
\begin{proof} The weak form  of the momentum equation on $u$ \eqref{E} reads:
\begin{equation*}
\text{for all } \varphi \in {\cal C}_c^\infty\left(\R_+ \times \Omega\right) \mbox{ with } \div \varphi = 0, \quad  \int_0^{\infty} \int_\Omega \left( u \cdot \pd_t \varphi +   (u \otimes u) : \nabla \varphi \right)  = -\int_\Omega u_0 \cdot \varphi(0, \cdot).
\end{equation*}
Considering $\varphi=\nabla^\perp \psi$ with $\psi \in {\cal D}\left([0, +\infty) \times \Omega\right)$ and integrating by parts in the above equation, we already know that the transport equation holds on $\R_+ \times \Om$. For any $\varepsilon,\rho>0$ we consider $\eta_{\varepsilon}$ and $\chi_{\rho}$ defined on $\R^2$ in the following way: $\chi_{\rho}$ is smooth, such that $\chi_\rho\equiv 1$ in $\cup_{k=1}^NB(x_{k},\rho)^c$ and $\chi_{\rho}\equiv 0$ in $\cup_{k=1}^N B(x_{k},\rho/2)$, and
$$\eta_\e(x)=\eta\left( \frac{1-|\Tc(x)|}{\e}\right)\quad \text{if }x\in \Omega,\quad \eta_\e(x)=0\quad \text{if }x\notin \Omega,$$where 
\begin{equation}\label{def:cut-off}
\eta: \R_+\to [0,1]\quad \text{is smooth, non-decreasing, }\quad \eta\equiv 0 \text{ on } [0,1/2],\quad
\eta \equiv 1 \text{  on } [1,+\infty).\end{equation}
 By Proposition \ref{prop T}, the functions $\eta_\e$ and  $\chi_\rho \eta_\e$ are  smooth and compactly supported in $\Omega$.

Let us fix $\varphi\in C^\infty_{c}( [0,\infty) \times \R^2)$ such that $\supp(\varphi)\subset [0,T]\times \R^2$. Then $\chi_{\rho}\eta_{\varepsilon}\varphi \in C^\infty_{c}( [0,\infty) \times \Om)$ and we have:
\begin{eqnarray*}
\int_{0}^\infty \int_{\Om} (\omega \partial_{t}(\chi_{\rho}\eta_{\varepsilon}\varphi) + u\omega \cdot \nabla (\chi_{\rho}\eta_{\varepsilon}\varphi))(t,x)\d x\d t &=& -  \int_{\Om} \omega_{0}(x) \chi_{\rho}\eta_{\varepsilon}(x)\varphi(0,x) \d x\end{eqnarray*}
therefore
\begin{eqnarray*}
\int_{0}^\infty \int_{\R^2} \chi_{\rho}\eta_{\varepsilon} (\bar\omega \partial_{t}\varphi + \bar u\bar \omega \cdot \nabla \varphi))(t,x)\d x\d t &=& -  \int_{\R^2} \chi_{\rho}\eta_{\varepsilon}(x)  \bar\omega_{0}(x)\varphi(0,x) \d x \\
&&- \int_{0}^\infty \int_{\R^2} \chi_{\rho}((\bar u\bar \omega \varphi) \cdot \nabla \eta_{\varepsilon})(t,x)\d x\d t \\
&&- \int_{0}^\infty \int_{\R^2}  \eta_{\varepsilon}((\bar u\bar \omega \varphi) \cdot \nabla\chi_{\rho})(t,x)\d x\d t.
\end{eqnarray*}
We fix $\rho>0$ and let $\varepsilon$ tend to $0$. We first deduce easily from the uniform estimates on $u$ and $\omega$ in $\R_+\times \Om$ that
\begin{equation}\label{lim:1}
\begin{split}
\int_{0}^\infty \int_{\R^2} \chi_{\rho}\eta_{\varepsilon} (\bar\omega \partial_{t}\varphi + \bar u\bar \omega \cdot \nabla \varphi)(t,x)\d x\d t
\longrightarrow&  \int_{0}^\infty \int_{\R^2} \chi_\rho (\bar\omega \partial_{t}\varphi + \bar u\bar \omega \cdot \nabla \varphi)(t,x)\d x\d t
\end{split}
\end{equation}
and
\begin{equation}\label{lim:2}
\int_{\R^2}\chi_{\rho}(x) \eta_{\varepsilon}(x)  \bar\omega_{0}(x)\varphi(0,x) \d x \to \int_{\R^2}  \chi_\rho(x) \bar\omega_{0}(x)\varphi(0,x) \d x.
\end{equation}

Next, by using uniform bounds on $\omega$ and $u$ (see Proposition \ref{prop:uniform-bound}), we have:
\begin{equation}\label{lim:3}
\begin{split}
\limsup_{\varepsilon \to 0}\Bigl| \int_{0}^\infty \int_{\R^2}  \eta_{\varepsilon}(\bar u \bar \omega \varphi) \cdot \nabla\chi_{\rho}(t,x)\d x\d t \Bigl| 
&\leq C T \|\varphi\|_{L^\infty} \|u\|_{L^\infty}\|\omega\|_{L^\infty}\|\nabla \chi_\rho\|_{L^1(\R^2)}\\
&\leq C T \|\varphi\|_{L^\infty} \|u\|_{L^\infty}\|\omega\|_{L^\infty}N\rho.
\end{split}
\end{equation}

We then claim that:
\begin{equation}\label{lim:4}
\lim_{\e \to 0}\int_{0}^\infty \int_{\R^2} \chi_{\rho}((\bar u\bar \omega \varphi) \cdot \nabla \eta_{\varepsilon})(t,x)\d x\d t=0.
\end{equation}

Indeed, we have for $x\in \Omega\setminus \cup_{k=1}^N B(x_k, \rho/2)$
$$\nabla \eta_\e(x)=-\frac{1}{\e}\eta'\left(\frac{1-|\Tc(x)|}{\e}\right) D\Tc^T(x) \frac{\Tc(x)}{|\Tc(x)|}$$
therefore 
$$u(t,x)\cdot \nabla \eta_\e(x)=-\frac{1}{\e}\eta'\left(\frac{1-|\Tc(x)|}{\e}\right) U(t,\Tc(x))\cdot  \frac{\Tc(x)}{|\Tc(x)|},$$
where $U$ is defined in Proposition \ref{log-lip}.

Setting $y=\Tc(x)$, we have
\begin{equation*}
\begin{split}
u(t,x)\cdot \nabla \eta_\e(x)=-\frac{1}{\e}
\eta'\left(\frac{1-|y|}{\e}\right) \Big[U(t,y)-U\left(t,\frac{y}{|y|}\right)\Big]\cdot  \frac{y}{|y|}-
\frac{1}{\e}
\eta'\left(\frac{1-|y|}{\e}\right)U\left(t,\frac{y}{|y|}\right)\cdot  \frac{y}{|y|}.\end{split}\end{equation*}
On the one hand, we have by virtue of Lemma \ref{est:K4}
\begin{equation*}
U\left(t,\frac{y}{|y|}\right)\cdot  \frac{y}{|y|}=0.
\end{equation*}
On the other hand, by Proposition \ref{log-lip} we have
$$\left|U(t,y)-U\left(t,\frac{y}{|y|}\right)\right|\leq Ch\left(1-|y|\right).$$
We therefore obtain
\begin{equation*}
\begin{split}
|u(t,x)\cdot \nabla \eta_\e(x)|\leq \frac{C}{\e}
\left|\eta'\left(\frac{1-|y|}{\e}\right)\right| h(\e).\end{split}\end{equation*}

Finally, changing variables yields
\begin{equation*}
\begin{split}
\int_{0}^\infty &\int_{\R^2} |\chi_{\rho}\bar \omega \varphi \bar u\cdot \nabla \eta_{\varepsilon})|
\d x\d t\\
&\leq C T h(\e)\|\varphi\|_{L^\infty}\| \omega\|_{L^\infty}\int_D \text{det }(D\Tc^{-1}(y) )\chi_\rho\left(\Tc^{-1}(y)\right)\frac{1}{\e}
\left|\eta'\left(\frac{1-|y|}{\e}\right)\right|\d y\\
&\leq C(\rho) T h(\e)\|\varphi\|_{L^\infty}\| \omega\|_{L^\infty}\int_{0}^1 \frac{r}{\e}
\left|\eta'\left(\frac{1-r}{\e}\right)\right|\d r\\
&\leq C(\rho)T h(\e)\|\varphi\|_{L^\infty}\| \omega\|_{L^\infty}\int_{0}^{1/\e} (1-\e \tau)
\left|\eta'\left(\tau\right)\right|\d \tau\leq C(\rho) T h(\e)\|\varphi\|_{L^\infty}\| \omega\|_{L^\infty}\|\eta'\|_{L^1(\R_+)},
\end{split}
\end{equation*}
where we have used the fact that $D\Tc^{-1}$ is smooth away from the corners together with the support properties of $\chi_\rho$. The claim \eqref{lim:4} follows.

\medskip

To conclude the proof, we let eventually $\rho$ tend to $0$ in \eqref{lim:1}, \eqref{lim:2} and \eqref{lim:3}.

\end{proof}

\medskip

\textbf{Proof of Proposition \ref{prop:lagrangian-formulation}}

In view of the assumptions \eqref{cond:u-div}-\eqref{bar u}, the results of DiPerna and Lions \cite{dip-lions} on linear transport equations ensure that $\bar \omega$ is the unique solution in $L^\infty(\R_+, L^{p'}(\R^2))$ to the linear transport equation with field $\bar u$ (where $p'$ denotes the conjugate exponent of $p$). For a precise statement, we refer to \cite[Theo. II.2]{dip-lions}. We also refer to, e.g., \cite{ambrosio} (Section 4) for more recent developments in the theory. On the other hand, for all $t\in \R_+$ we have $X(t,\Omega)\subset \Omega$; hence recalling that $X(t,\cdot)$ preserves Lebesgue's measure we observe that by Fubini theorem, for a.e. $(t,x)\in \R_+\times  \Omega$,  $u(s,X(s,x))=\overline{u}(s,X(s,x))$. Hence given the definition of $X$ one can readily prove that the map $\tilde{\omega}(t):=X(t,\cdot)_\#\bar \omega_0$ is a solution in $L^\infty(\R_+,L^{p'}(\R^2))$ to the same linear transport equation with field $\bar u$ (see e.g. the proof of Proposition 2.1 in \cite{ambrosio}). By uniqueness, we conclude that $\bar \omega(t)=\tilde{\omega}(t)=X(t,\cdot)_\#\bar \omega_0$ for a.e. $t\in \R_+$, as we wanted.

\medskip

\subsection{Uniqueness by a Lagrangian approach}

We consider two solutions $(u_1,\omega_1)$ and $(u_2,\omega_2)$ of \eqref{E}-\eqref{E2} with the same initial datum and denote by $X_k$, $k=1,2$, the corresponding flows given by Proposition \ref{prop:lagrangian-formulation}. 

\medskip

As
\begin{equation*}
\begin{split}
|\mathcal{T}(X_1(s,x))-\mathcal{T}(X_2(s,x)) | 
&\leq  \|D\Tc \|_{L^\infty}  \int_0^s  |u_1(\tau,X_1(\tau,x))-u_2(\tau,X_2(\tau,x)) | \d \tau \\
&\leq s\|D\Tc \|_{L^\infty} (\|u_1\|_{L^\infty}+\|u_2\|_{L^\infty}),
\end{split}
\end{equation*}
we infer from \eqref{est:DT2} and Proposition \ref{prop:uniform-bound} that there exists $t_0$ depending only on $\Omega$ and $\|\omega\|_{L^\infty(\Omega)}$ such that 
 \begin{equation}\label{petitesse}
\sup_{s\in [0,t_0]}\sup_{x\in \Omega}|\mathcal{T}(X_1(s,x))-\mathcal{T}(X_2(s,x)) |<\min(|\Omega|^{-1},1).
\end{equation}

We define next for $t\in \R_+$
\begin{equation*}
f(t)=\int_0^t\int_{\Om}|\mathcal{T}(X_1(s,x))-\mathcal{T}(X_2(s,x))| \d x\d s.
\end{equation*}
We will show that $f$ vanishes identically on $[0,t_0]$.

Given the time regularity of $\mathcal{T}\circ X_i$, we can compute for all $t\in \R_+$
\begin{eqnarray*}
f'(t) &=&\int_{\Om}|\mathcal{T}(X_1(t,x))-\mathcal{T}(X_2(t,x))| \d x\\
&\leq & \int_{\Om} \int_0^t \Bigl| D\Tc(X_1(s,x)) \frac{\d}{\d t} X_1(s,x)-D\Tc(X_2(s,x)) \frac{\d}{\d s} X_2(s,x)\Bigl|\d s\d x\\
&\leq&\int_{\Om}\int_0^t |U_1(s,\Tc (X_1(s,x))-U_2(s,\Tc (X_2(s,x))| \d s \d x\\
&\leq& \int_0^t  \int_{\Om}|U_1(s,\Tc (X_1(s,x))-U_1(s,\Tc (X_2(s,x))|  \d x \d s\\
&&\quad+\int_0^t \int_{\Om} |U_1(s,\Tc (X_2(s,x))-U_2(s,\Tc (X_2(s,x))| \d x \d s\\
&=:& F_1+F_2.
\end{eqnarray*}
By  Proposition \ref{log-lip}, we have 
\[
F_1\leq C \int_0^t \int_\Om h(|\mathcal{T}(X_1(s,x))-\mathcal{T}(X_2(s,x))|) \d x\d s,
\]
where we recall $h(r)=r(1+|\ln r|)$.

Next, since $X_2(s,\cdot)$ preserves the Lebesgue's measure on $\Omega$, we have
\begin{equation*}
F_2=\int_0^t \int_{\Om}|U_1(s,\Tc (x))-U_2(s,\Tc (x))| \d x\d s=\int_0^t \int_{\Om}\left|D\Tc(x)(
u_1(s,x)-u_2(s,x))\right| \d x\d s.
\end{equation*}
Let $s\in \R_+$ such that $\omega_i(s)=X_i(s)_\#\omega_0$. Let $x\in \Omega$ and $\varepsilon>0$ such that $\min_k |T(x)-T(x_k)|>\varepsilon.$ Recalling that $\eta$ is the cut-off function defined in \eqref{def:cut-off}, we have by continuity of $\tilde{x}\mapsto K_D(\Tc(x),\Tc(\tilde{x}))\eta\left(\frac{|\Tc(x)-\Tc(\tilde{x})|}{\varepsilon}\right)$
\begin{equation*}\begin{split}
 \int_{\Om}K_D(\Tc(x),\Tc(\tilde{x}))&\eta\left(\frac{|\Tc(x)-\Tc(\tilde{x})|}{\varepsilon}\right)\omega_i(s,\tilde{x})\d \tilde{x}\\&= \int_{\Om}K_D(\Tc(x),\Tc(X_i(s,\bar{x})))\eta\left(\frac{|\Tc(x)-\Tc(X_i(s,\bar{x}))|}{\varepsilon}\right)\omega_0(\bar{x})\d \bar{x}.
\end{split}\end{equation*}
Letting $\varepsilon$ go to $0$ we can use again the previous estimates on $K_D$ and $D\Tc$ so that all terms above pass to the limit,  and the Biot-Savart law yields
\[
u_i(s,x)=D\Tc(x)^{T} \int_{\Om}K_D(\Tc(x),\Tc(\tilde{x}))\omega_i(s,\tilde{x})\d \tilde{x}=
D\Tc(x)^{T} \int_{\Om}K_D(\Tc(x),\Tc(X_i(s,\bar{x}))\omega_0(\bar{x})\d \bar{x}.
\]
Thus bringing this latter into $F_2$ and changing variables $y=\Tc(x)$, we have
\begin{eqnarray*}
F_2&\leq&  \int_0^t \int_\Om\det (D\mathcal{T}(x))\left(\int_\Om \Big|K_D(\mathcal{T}(x),\mathcal{T}(X_1(s,\bar x)))-K_D(\mathcal{T}(x),\mathcal{T}(X_2(s,\bar x)))\Big| |\omega_0(\bar{x}) | \d \bar x\right) \d x\, \d s\\
&=&\int_0^t \int_D\left(\int_\Om \Big|K_D(y,\mathcal{T}(X_1(s,\bar x)))-K_D(y,\mathcal{T}(X_2(s,\bar x)))\Big|  |\omega_0(\bar{x})| \d \bar x\right) \d y\d s\\
&\leq&  \| \omega_0\|_{L^\infty(\Omega)}\int_0^t \int_\Om\left(\int_{D}  |K_D(y,\mathcal{T}(X_1(s,\bar x)))-K_D(y,\mathcal{T}(X_2(s,\bar x)))| \d y\right) \d \bar x\d s\\
&\leq& C \int_0^t \int_\Om h(|\mathcal{T}(X_1(s,\bar x))-\mathcal{T}(X_2(s, \bar x))|)\d \bar x\d s,
\end{eqnarray*}
where we have applied Lemma \ref{est:K3} in the last inequality.\label{K3bis}

For $t\in [0,t_0]$ we finally apply Jensen's inequality in the estimates for $F_1$ and $F_2$ using that $h$ is concave on $[0,1]$ and \eqref{petitesse} to  obtain
\[
f'(t)\leq C  h(f(t)), \quad \forall t\in [0,t_0].
\]
Therefore $f(t)\leq f(0)^{\exp(-Ct)} e^{1-\exp(-Ct)}=0$ on $[0,t_0]$. Hence for a.e. $x\in \Omega$ we have $X_1(t,x)=X_2(t,x)$ on $[0,t_0]$. Repeating the argument on the intervals $[kt_0,(k+1)t_0]$, $k\in \mathbb{N}$ we conclude that $\omega_1=\omega_2$ a.e.  on $\R_+$ and Theorem \ref{thm:main} is proved.

\section{Proof of Proposition \ref{log-lip}} \label{sect 5}

In this last section, we write all the details to establish the log-Lipschitz regularity of the vector field $U$.
By Proposition \ref{prop:uniform-bound} and \eqref{est:DT1}, we know that $U$ is uniformly bounded on $D$. Therefore it suffices to establish the inequality \eqref{ineq:loglip} when $$d:=|y_1-y_2|\leq \frac{\delta}{2}<1$$
(we recall that $\delta>0$ is defined in Proposition \ref{prop T}).
For two subsets $\Sigma_1 \subset \Sigma_2\subset D$ we will split $U(y_1)-U(y_2)$ using Remark \ref{DTDTt} and \eqref{eq:BS-2}:
\begin{equation}\label{eq:decomposition-1}
\begin{split}
U(y_1)-U(y_2)=&\det \big(D\mathcal{T}(\Tc^{-1}(y_{1}))\big)\int_{D} K_D(y_1,z)\omega(\mathcal{T}^{-1}(z))\det (D\mathcal{T}^{-1}(z))\d z \\
&- \det \big(D\mathcal{T}(\Tc^{-1}(y_{2}))\big)\int_{D} K_D(y_2,z)\omega(\mathcal{T}^{-1}(z))\det (D\mathcal{T}^{-1}(z))\d z\\
=:&[U_1-U_2]+[V_1+V_2]+[W_1+W_2],
\end{split}
\end{equation}
where
\begin{equation}\label{eq:decomposition-2}
\begin{split}
U_1&=\det \big(D\mathcal{T}(\Tc^{-1}(y_1))\big)\int_{\Sigma_1} K_D(y_1,z)\omega(\mathcal{T}^{-1}(z))\det (D\mathcal{T}^{-1}(z))\d z\\
U_2&=\det \big(D\mathcal{T}(\Tc^{-1}(y_2))\big)\int_{\Sigma_1} K_D(y_2,z)\omega(\mathcal{T}^{-1}(z))\det (D\mathcal{T}^{-1}(z))\d z,
\end{split}
\end{equation}
\begin{equation}\label{eq:decomposition-3}
\begin{split}
V_1=&\Bigl(\det \big(D\mathcal{T}(\Tc^{-1}(y_1))\big)-\det \big(D\mathcal{T}(\Tc^{-1}(y_2))\big)\Bigl)\int_{D\setminus \Sigma_2} K_D(y_1,z)\omega(\mathcal{T}^{-1}(z))\det (D\mathcal{T}^{-1}(z))\d z\\
V_2=&\det \big(D\mathcal{T}(\Tc^{-1}(y_2))\big)\int_{D\setminus \Sigma_2} \Bigl(K_D(y_1,z)-K_D(y_2,z)\Bigl)\omega(\mathcal{T}^{-1}
(z))\det (D\mathcal{T}^{-1}(z))\d z\\
\end{split}
\end{equation}
and
\begin{equation}\label{eq:decomposition-4}
\begin{split}
W_1=&\Bigl(\det \big(D\mathcal{T}(\Tc^{-1}(y_1))\big)-\det \big(D\mathcal{T}(\Tc^{-1}(y_2))\big)\Bigl)\int_{\Sigma_2\setminus \Sigma_1} K_D(y_1,z)\omega(\mathcal{T}^{-1}(z))\det (D\mathcal{T}^{-1}(z))\d z\\
W_2=&\det \big(D\mathcal{T}(\Tc^{-1}(y_2))\big)\int_{\Sigma_2\setminus \Sigma_1} \Bigl(K_D(y_1,z)-K_D(y_2,z)\Bigl)\omega(\mathcal{T}^{-1}(z))\det (D\mathcal{T}^{-1}(z))\d z.
\end{split}
\end{equation}

\medskip

\noindent \textbf{First step:} $y_1\in D\setminus (\cup_{k=1}^N B(\Tc(x_{k}),\delta))$. It follows that $y_2\in D\setminus (\cup_{k=1}^N B(\Tc(x_{k}),\delta/2))$.

We set $\Sigma_1=\Sigma_2=\emptyset$ in \eqref{eq:decomposition-1}, so that $U_1=U_2=0=W_1=W_2$. Next, since  $\text{det}(D\mathcal{T}\circ\mathcal{T}^{-1})$ is smooth and its derivative is bounded away from the points $\Tc(x_{k})$ we have
$$|V_1|\leq C|y_1-y_2|\int_{D} \left|
K_D(y_1,z)\right||\omega(\mathcal{T}^{-1}(z))|\det (D\mathcal{T}^{-1}(z))\d z\leq Cd\|\omega\|_{ L^\infty(\Om)},$$
where we have used \eqref{estimate:K1} to estimate the integral.

For the last term, we have by \eqref{est:DT2}
\begin{equation*}
\begin{split}
|V_2|&\leq C\sum_{k=1}^N \int_{D\cap B(\Tc(x_{k}),\delta/4)}\left|K_D(y_1,z)-K_D(y_2,z)\right||\omega(\mathcal{T}^{-1}(z))|\det (D\mathcal{T}^{-1}(z))\d z\\
&+C\int_{D\setminus (\cup_{k=1}^N B(\Tc(x_{k}),\delta/4))}\left|K_D(y_1,z)-K_D(y_2,z)\right||\omega(\mathcal{T}^{-1}(z))|\det (D\mathcal{T}^{-1}(z))\d z\\
&\leq I_1+I_2.
\end{split}
\end{equation*}
On the one hand, for $z\in B(\Tc(x_{k}),\delta/4)$ we have $|y_1-z|\geq 3\delta/4$ and $|y_2-z|\geq \delta/4$, hence \eqref{estimate:K2} yields
\begin{equation*}
I_1 \leq \frac{16C}{3\delta^2}
|y_1-y_2|\int_{D}|\omega(\mathcal{T}^{-1}(z))|\det (D\mathcal{T}^{-1}(z))\d z \leq Cd \|\omega\|_{L^1(\Om)}.
\end{equation*}
On the other hand,  recall that $D\mathcal{T}^{-1}$ is bounded away from $\Tc(x_{k})$, so that applying Lemma \ref{est:K3} we find
\begin{equation*}
I_2\leq C\|\omega\|_{L^\infty(\Omega)}
\int_{D}\left|K_D(y_1,z)-K_D(y_2,z)\right|\d z\leq C h(|y_{1}-y_{2}|)\|\omega\|_{L^\infty(\Omega)}.
\end{equation*}

So combining the previous estimates we obtain
\begin{equation*}
|U(y_1)-U(y_2)|\leq C \|\omega\|_{L^\infty(\Omega)} h(|y_1-y_2|),\quad \forall y_1\in D\setminus (\cup_{k=1}^N B(\Tc(x_{k}),\delta)).
\end{equation*}

By symmetry , we also have
\begin{equation*}
|U(y_1)-U(y_2)|\leq C \|\omega\|_{L^\infty(\Omega)} h(|y_1-y_2|),\quad \forall y_2\in D\setminus (\cup_{k=1}^N B(\Tc(x_{k}),\delta)).
\end{equation*}

\medskip

\noindent \textbf{Second step:} $y_1,y_{2}\in B(\Tc(x_{k}),\delta)$.

We note that $g:=\det (D\mathcal{T}\circ\Tc^{-1}))$ satisfies  $g(y)=\mathcal{O}(|y-\Tc(x_{k})|^{2\alpha_{k}})$ in the neighborhood of the points $\Tc(x_{k})$. By Proposition \ref{prop T} and  \eqref{est:DT1}, \eqref{est:DT1-bis}, \eqref{est:DT2} we estimate
\begin{equation*}
\begin{split}
|\nabla g(y)| &\leq 4 |D\mathcal{T}(\Tc^{-1}(y))| |D^2\mathcal{T}(\Tc^{-1}(y))| |D\Tc^{-1}(y)| 
\leq C |y - \Tc(x_{k})|^{2\alpha_{k}-1},
\end{split}
\end{equation*}
hence, we infer by the mean value theorem that
\begin{equation}\label{MVT}
| \det \big(D\mathcal{T}(\Tc^{-1}(y_{1}))\big) -  \det \big(D\mathcal{T}(\Tc^{-1}(y_{2}))\big) | \leq Cd \sup_{y\in [y_{1},y_{2}]} |y-\Tc(x_{k})|^{2\alpha_{k}-1}
\end{equation}
with $\alpha_{k}\geq 1/2$. We set $$\Sigma_2=D\cap B(\Tc(x_{k}),\delta).$$
By \eqref{eq:decomposition-3} we have, using \eqref{MVT}, \eqref{estimate:K1}, the fact that $D\mathcal{T}^{-1}$ is bounded away from the $x_{j}$ and that $|y_{1}-z|\geq \delta$ when $z\in B(\Tc(x_{j}),\delta)$ for $j\neq k$:
\begin{equation*}
\begin{split}
|V_1|\leq Cd \|\omega\|_{L^\infty(\Omega)}\Bigl(& \int_{D\setminus (\cup_{j=1}^N B(\Tc(x_{j}),\delta)) } |K_D(y_1,z)| \det (D\mathcal{T}^{-1}(z)) \d z \\
 &+ \sum_{j\neq k}\int_{ B(\Tc(x_{j}),\delta)\cap D} |K_D(y_1,z)|\det (D\mathcal{T}^{-1}(z)) \d z \Bigl)\\
 \leq Cd \|\omega\|_{L^\infty(\Omega)}\Bigl(& C \int_{D} |y_{1} -z|^{-1}  \d z + \frac C{\delta} \sum_{j\neq k}\int_{ D} \det (D\mathcal{T}^{-1}(z)) \d z \Bigl)\\
  \leq  Cd\|\omega\|_{L^\infty(\Omega)}.&
\end{split}
\end{equation*}
Similarly, Lemma \ref{est:K3} and  \eqref{estimate:K2} yield
\begin{equation*}
\begin{split}
|V_2| \leq C \|\omega\|_{L^\infty(\Omega)}\Bigl(& \int_{D\setminus (\cup_{j=1}^N B(\Tc(x_{j}),\delta)) }\left|K_D(y_1,z)-K_D(y_2,z)\right| \det (D\mathcal{T}^{-1}(z)) \d z \\
 &+ \sum_{j\neq k}\int_{ B(\Tc(x_{j}),\delta)\cap D} \left|K_D(y_1,z)-K_D(y_2,z)\right| \det (D\mathcal{T}^{-1}(z)) \d z \Bigl)\\
 \leq C \|\omega\|_{L^\infty(\Omega)}\Bigl(& C \int_{D} \left|K_D(y_1,z)-K_D(y_2,z)\right|  \d z + \frac{ Cd}{\delta^2} \sum_{j\neq k}\int_{ D} |\det (D\mathcal{T}^{-1}(z))| \d z \Bigl)\\
\leq C\|\omega\|_{L^\infty(\Omega)}(&  h(d) +  d).
\end{split}
\end{equation*}

It remains to estimate the parts $U_1-U_2$ and $W_1+W_2$ for a judicious choice of $\Sigma_1\subset D\cap B(\Tc(x_{k}),\delta)$.

As in the proof of Lemma \ref{est:K3}, we introduce
$$\tilde y:=\f{y_{1}+y_{2}}{2}$$
and we consider the following cases.

\medskip

{\bf Case 1:} $\Tc(x_{k})\in B(\tilde y, 5d)$.

We set $$\Sigma_1=D\cap B(\Tc(x_{k}),\delta)\cap B(\tilde y,6d)=\Sigma_2\cap B(\tilde y,6d).$$
Using the estimates of Proposition \ref{prop T} and \eqref{est:DT1} in the neighborhood  $\Tc(x_{k})$ we get for $i=1,2$
\begin{equation*}
\begin{split}
|U_i|
&\leq C\|\omega\|_{L^\infty(\Omega)} |\Tc(x_{k})-y_i|^{2\alpha_k}\int_{\Sigma_1} |K_D(y_i,z)||z-\Tc(x_{k})|^{-2\alpha_k}\d z.
\end{split}
\end{equation*}
Next, observing that
\[
|\Tc(x_{k})-y_i|^{2\alpha_k} \leq 2^{2\alpha_k} (\max\{|y_i-z|,  |z-\Tc(x_{k})|\} )^{2\alpha_k} 
\]
we compute using \eqref{estimate:K1}
\begin{equation*}
\begin{split}
 |\Tc(x_{k})&-y_i|^{2\alpha_k}\int_{\Sigma_1} |K_D(y_i-z)||z-\Tc(x_{k})|^{-2\alpha_k}\d z\\
\leq& C |\Tc(x_{k})-y_i|^{2\alpha_k}\int_{\Sigma_1\cap \{|y_i-z|\leq |z-\Tc(x_{k})|\}} |y_i-z|^{-1}|z-\Tc(x_{k})|^{-2\alpha_k}\d z\\
&+C|\Tc(x_{k})-y_i|^{2\alpha_k}\int_{\Sigma_1\cap \{|y_i-z|\geq |z-\Tc(x_{k})|\}} |y_i-z|^{-1}|z-\Tc(x_{k})|^{-2\alpha_k}\d z\\
\leq& C \int_{\Sigma_1\cap \{|y_i-z|\leq |z-\Tc(x_{k})|\}} |y_i-z|^{-1}\d z
   +\int_{\Sigma_1\cap \{|y_i-z|\geq |z-\Tc(x_{k})|\}} |y_i-z|^{2\alpha_k-1}|z-\Tc(x_{k})|^{-2\alpha_k}\d z\\
   \leq& C \int_{ \{|y_i-z|\leq 7d\}} |y_i-z|^{-1}\d z
   +(7d)^{2\alpha_k-1} \int_{ \{|z-\Tc(x_{k})|\leq 7d\}} |z-\Tc(x_{k})|^{-2\alpha_k}\d z\\
 \leq &Cd.
\end{split}
\end{equation*}

This yields
\begin{equation*}
|U_1|+|U_2|\leq Cd\|\omega\|_{L^\infty(\Omega)}.
\end{equation*}

We next estimate $W_1$ and $W_2$. As $2\alpha_{k}-1\geq 0$, we have by \eqref{MVT}
\begin{equation}
\label{ineq:dt}
\Bigl|\det \big(D\mathcal{T}(\Tc^{-1}(y_1))\big)-\det \big(D\mathcal{T}(\Tc^{-1}(y_2))\big)\Bigl|\leq Cd
\big( |\Tc(x_{k})-y_1|^{2\alpha_k-1}+ |\Tc(x_{k})-y_2|^{2\alpha_k-1}\big).
\end{equation}
Hence
\begin{equation*}
\begin{split}
|W_1|&\leq C \|\omega\|_{L^\infty(\Omega)} d \big( |\Tc(x_{k})-y_1|^{2\alpha_k-1}+ |\Tc(x_{k})-y_2|^{2\alpha_k-1}\big)
 \int_{\Sigma_{2}\setminus\Sigma_1} |K_D(y_1,z)||z-\Tc(x_{k})|^{-2\alpha_k}\d z\\
&\leq C \|\omega\|_{L^\infty(\Omega)} d \big( |\Tc(x_{k})-y_1|^{2\alpha_k-1}+ |\Tc(x_{k})-y_2|^{2\alpha_k-1}\big)
 \int_{D\setminus\Sigma_1} |y_1-z|^{-1}|z-\Tc(x_{k})|^{-2\alpha_k}\d z.
\end{split}
\end{equation*}
Now, for $i=1,2$ and $z\in D\setminus \Sigma_1$ we have
$|\Tc(x_{k})-y_i|\leq |\Tc(x_{k})-\tilde y|+|\tilde y-y_i|\leq 6 d\leq |z-\tilde y|$. On the other hand, as $|\tilde y-\Tc(x_{k})|\leq 5d\leq 5|z-\tilde y|/6$ we get
$|z-\Tc(x_{k})|\geq |z-\tilde y|-|\Tc(x_{k})-\tilde y|\geq |z-\tilde y|/6$ and $|y_i-z|\geq |\tilde y-z|-|y_{i}-\tilde y|\geq 5|z-\tilde y|/6$. Therefore
\begin{equation*}
\begin{split}
|W_1|&\leq C  \|\omega\|_{L^\infty(\Omega)} d
 \int_{D\setminus B(\tilde y,6d)} |z-\tilde y|^{-2}\d z\leq C \|\omega\|_{L^\infty(\Omega)} d(1+ |\ln d|).
\end{split}
\end{equation*}

On the other hand, using again the estimates in the neighborhood of $\Tc(x_{k})$ and the estimate
\eqref{estimate:K2} we obtain
\begin{equation*}
\begin{split}
|W_2|&\leq C \|\omega\|_{L^\infty(\Omega)} d \,|\Tc(x_{k})-y_2|^{2\alpha_k}\int_{D \setminus\Sigma_1} |z-y_1|^{-1}|z-y_2|^{-1}|z-\Tc(x_{k})|^{-2\alpha_k}\d z.
\end{split}
\end{equation*}
Then, using the same inequalities as above for $z\in D\setminus \Sigma_1$, we get
\begin{equation*}
\begin{split}
|W_2| \leq C  \|\omega\|_{L^\infty(\Omega)}d
 \int_{D\setminus B(\tilde y,6d)} |z-\tilde y|^{-2}\d z\leq C \|\omega\|_{L^\infty(\Omega)} d(1+ |\ln d|).
\end{split}
\end{equation*}

\medskip

{\bf Case 2:} $\Tc(x_{k})\notin B(\tilde y,5d)$.

In particular  $ |y_i-\Tc(x_{k})|\geq |\Tc(x_{k})-\tilde y|-|\tilde y-y_i|\geq 4d$ for $i=1,2$, so that
\begin{equation}\label{equi:y}\frac{|y_2-\Tc(x_{k})|}{2}\leq |y_1-\Tc(x_{k})|\leq 2|y_2-\Tc(x_{k})|.\end{equation}
We set
$$\Sigma_1=D\cap B(\Tc(x_{k}),\delta)\cap B(\tilde y,2d)=\Sigma_2\cap B(\tilde y,2d) .$$
For $z\in \Sigma_1$ we have $|z-\Tc(x_{k})|\geq |\tilde y-\Tc(x_{k})|-|z-\tilde y|\geq 3d$, therefore
\begin{equation*}
\begin{split}
|U_i|
\leq &C\|\omega\|_{L^\infty(\Omega)} |\Tc(x_{k})-y_i|^{2\alpha_k}\int_{\Sigma_1}
|y_i-z|^{-1}|z-\Tc(x_{k})|^{-2\alpha_k}\d z\\
\leq &C\|\omega\|_{L^\infty(\Omega)} \Bigl(\int_{\Sigma_1\cap \{|y_i-z|\leq |z-\Tc(x_{k})|\}} (2|\Tc(x_{k})-z|)^{2\alpha_k}
|y_i-z|^{-1}|z-\Tc(x_{k})|^{-2\alpha_k}\d z    \\
&+\int_{\Sigma_1\cap \{|y_i-z|\geq |z-\Tc(x_{k})|\}}  (2|z-y_i|)^{2\alpha_k}
|y_i-z|^{-1}|z-\Tc(x_{k})|^{-2\alpha_k}\d z   \Bigl)        \\
\leq& C \|\omega\|_{L^\infty(\Omega)} (2^{2\alpha_k} + (6d)^{2\alpha_k}(3d)^{-2\alpha_k} ) \int_{B(y_i,3d)}|y_i-z|^{-1}\d z=C\|\omega\|_{L^\infty(\Omega)} d.
\end{split}
\end{equation*}

To estimate $W_1+W_2$ we further decompose $\Sigma_2\setminus \Sigma_1$ in two sets, setting
$$\Sigma_2\setminus \Sigma_1=
\Big[\Sigma_2\cap B(\Tc(x_{k}),2d)\Big]\bigcup \Big[\Sigma_2\setminus [B(\Tc(x_{k}),2d)\cup B(\tilde y,2d)]\Big].$$
By \eqref{ineq:dt} and \eqref{equi:y} we have
\begin{equation*}
\begin{split}
|W_1+W_2|\leq& C \|\omega\|_{L^\infty(\Omega)} d  |\Tc(x_{k})-y_1|^{2\alpha_k-1}
 \int_{B(\Tc(x_{k}),2d)} |y_1-z|^{-1}|z-\Tc(x_{k})|^{-2\alpha_k}\d z\\
&+C \|\omega\|_{L^\infty(\Omega)}d \,|\Tc(x_{k})-y_1|^{2\alpha_k}\int_{B(\Tc(x_{k}),2d)} |z-y_1|^{-1}|z-y_2|^{-1}|z-\Tc(x_{k})|^{-2\alpha_k}\d z\\
&+C \|\omega\|_{L^\infty(\Omega)} d |\Tc(x_{k})-y_1|^{2\alpha_k-1}
 \int_{D\setminus [B(\Tc(x_{k}),2d)\cup B(\tilde y,2d)]} |y_1-z|^{-1}|z-\Tc(x_{k})|^{-2\alpha_k}\d z\\
&+C \|\omega\|_{L^\infty(\Omega)} d \,|\Tc(x_{k})-y_1|^{2\alpha_k}\int_{D\setminus [B(\Tc(x_{k}),2d)\cup B(\tilde y,2d)]} |z-y_1|^{-1}|z-y_2|^{-1}|z-\Tc(x_{k})|^{-2\alpha_k}\d z.
\end{split}
\end{equation*}

We first estimate the contributions of the integrals on the domain $ B(\Tc(x_{k}),2d)$.
For $z\in B(\Tc(x_{k}),2d)$ we have $|y_i-z|\geq |y_{i}-\Tc(x_{k})|-|z-\Tc(x_{k})| \geq  2d\geq  |z-\Tc(x_{k})|$. Therefore $|y_1-\Tc(x_{k})|\leq |z-\Tc(x_{k})|+|y_1-z|\leq 2|y_1-z|$.
It follows that
\begin{equation*}
\begin{split}
 |\Tc(x_{k})-y_1|^{2\alpha_k-1}
 \int_{B(\Tc(x_{k}),2d)} |y_1-z|^{-1}&|z-\Tc(x_{k})|^{-2\alpha_k}\d z\\
 &\leq C \int_{B(\Tc(x_{k}),2d)} |y_1-z|^{2\alpha_k-2}|z-\Tc(x_{k})|^{-2\alpha_k}\d z.
\end{split}
\end{equation*}
Similarly, using that $ |z-y_1| \leq  |z-y_2| +|y_{1}-y_{2}| \leq   |z-y_2| +d \leq 2 |z-y_2|$, we obtain
\begin{equation*}
\begin{split}
|\Tc(x_{k})-y_1|^{2\alpha_k}\int_{B(\Tc(x_{k}),2d)} |z-y_1|^{-1}&|z-y_2|^{-1}|z-\Tc(x_{k})|^{-2\alpha_k}\d z
\\
&\leq C \int_{B(\Tc(x_{k}),2d)}|y_1-z|^{2\alpha_k-2}|z-\Tc(x_{k})|^{-2\alpha_k}\d z.
\end{split}
\end{equation*}

Then we finally obtain, using that $2\alpha_k-2< 0$ and $|y_1-z|\geq 2d$,
\begin{equation*}
\begin{split}
 \int_{B(\Tc(x_{k}),2d)}|y_1-z|^{2\alpha_k-2}|z-\Tc(x_{k})|^{-2\alpha_k}\d z\leq C d ^{2\alpha_k-2}\int_{B(\Tc(x_{k}),2d)}
|z-\Tc(x_{k})|^{-2\alpha_k}\d z\leq C.
\end{split}
\end{equation*}

We now turn to the contributions of the integrals on the last domain $\Sigma_3:= D\setminus [B(\Tc(x_{k}),2d)\cup B(\tilde y,2d)]$.

We have on the one hand:
\begin{equation*}
\begin{split}
|\Tc(x_{k})-y_1|^{2\alpha_k-1}
 \int_{\Sigma_3} &|y_1-z|^{-1}|z-\Tc(x_{k})|^{-2\al_{k}}\d z\\
\leq& \int_{\Sigma_3\cap \{|y_1-z|\leq|z-\Tc(x_{k})| \}}  (2|z-\Tc(x_{k})|)^{2\alpha_k-1} |z-\Tc(x_{k})|^{-2\alpha_k}|y_1-z|^{-1}\d z\\
&  + \int_{\Sigma_3\cap \{|y_1-z|\geq|z-\Tc(x_{k})| \}} (2|y_1-z|)^{2\alpha_k-1} |z-\Tc(x_{k})|^{-2\alpha_k}|y_1-z|^{-1}\d z\\
\leq& C\int_{\Sigma_3\cap \{|y_1-z|\leq|z-\Tc(x_{k})| \}}  |z-\Tc(x_{k})|^{-1} |y_1-z|^{-1}\d z\\
&  +C \int_{\Sigma_3\cap \{|y_1-z|\geq|z-\Tc(x_{k})| \}} |y_1-z|^{2\alpha_k-2} |z-\Tc(x_{k})|^{-2\alpha_k}\d z\\
\leq& C\int_{D\setminus B(y_{1},d)}   |y_1-z|^{-2}\d z  +C \int_{D\setminus B(\Tc(x_{k}),2d)}  |z-\Tc(x_{k})|^{-2}\d z\\
\leq& C (1+ |\ln d|).
\end{split}
\end{equation*}
On the other hand, we use $|y_{1}-z|\leq d + |y_{2}-z| \leq 2 |y_{2}-z|$ and conversely $|y_{2}-z|\leq 2 |y_{1}-z|$ to compute
\begin{equation*}
\begin{split}
 |\Tc(x_{k})-y_1|^{2\alpha_k}\int_{\Sigma_3}& |z-y_1|^{-1}|z-y_2|^{-1}|z-\Tc(x_{k})|^{-2\alpha_k}\d z\\
\leq&\int_{\Sigma_3\cap \{|y_1-z|\leq|z-\Tc(x_{k})| \}}  (2|z-\Tc(x_{k})|)^{2\alpha_k} |z-y_1|^{-1}|z-y_2|^{-1} |z-\Tc(x_{k})|^{-2\alpha_k}\d z\\
& +   \int_{\Sigma_3\cap \{|y_1-z|\geq|z-\Tc(x_{k})| \}} (2|y_1-z|)^{2\alpha_k} |z-y_1|^{-1}|z-y_2|^{-1} |z-\Tc(x_{k})|^{-2\alpha_k} \d z\\
\leq & C\int_{D\setminus B(y_{1},d)}|y_1-z|^{-2}\d z +C\int_{D\setminus B(\Tc(x_{k}),2d)}  |z-\Tc(x_{k})|^{-2}\d z\\
&\leq C (1+|\ln d|).
\end{split}
\end{equation*}

Combining all the above estimates, we get the required results.

\bigskip

\noindent
{\bf Acknowledgement:} The first and the third authors are partially supported by the Project ``Instabilities in Hydrodynamics'' funded by Paris city hall (program ``Emergences'') and the Fondation Sciences Math\'ematiques de Paris. The second author is supported by the ANR projects GEODISP ANR-12-BS01-0015-01  and SCHEQ ANR-12-JS01-0005-01.

The authors are grateful to the anonymous referee for his/her valuable comments on the first version of this article which led to a substantial improvement of this work.

\adrese

\end{document}